\documentclass[a4paper,12pt,reqno]{amsart}
\usepackage{amsfonts}
\usepackage{amsmath}
\usepackage{amssymb}
\usepackage{lipsum}
\usepackage{stackengine}
\usepackage{xcolor}
\usepackage{mathrsfs}
\usepackage{hyperref}
\usepackage{enumerate}
\numberwithin{equation}{section}
\usepackage{graphicx}
\usepackage{enumitem}

\newtheorem{theorem}{Theorem}[section]
\newtheorem{defi}[theorem]{Definition}
\newtheorem{remark}[theorem]{Remark}
\newtheorem{Cor}[theorem]{Corollary}
\newtheorem{prop}[theorem]{Proposition}
\newtheorem{lemma}[theorem]{Lemma}

\def\Rr{{\mathbb R}}

\def\Rn{{\mathbb R}^N}
\def\R2n{{\mathbb R}^{2N}}
\def\l{\lambda}
\def\pl{(-\Delta)_p^s}
\def\cq{\displaystyle\int\limits_{\Omega}\frac{|u(y)|^r}{|x-y|^{\mu}}|u(x)|^{r-2}u(x)dy}
\def\O{\Omega}
\def\w{W_0^{s,p}(\O)}
\def\a{\alpha}
\def\p{(\textbf{P})}
\def\cq{\displaystyle\left(\int\limits_\O\frac{|u(y)|^r}{|x-y|^\a}dy\right)|u(x)|^{r-2}u(x)}

\def\N{{\mathbb N}}

\usepackage{fancyhdr}
\begin{document}
	\title[Some existence and uniqueness results for logistic Choquard equations]
	{Some existence and uniqueness results for logistic Choquard equations}
	\author[G.C.Anthal]{G.C.Anthal}
	\address{
		Gurdev Chand Anthal:	
		\endgraf
		Department of Mathematics
		\endgraf
		Indian Institute of Technology, Delhi, Hauz Khas
		\endgraf
		New Delhi-110016 
		\endgraf
		India
		\endgraf
		{\it E-mail address} {\rm Gurdevanthal92@gmail.com}
	}
	\author[J. Giacomoni]{J. Giacomoni}
	\address{
		J. Giacomoni:	
		\endgraf
		LMAP (UMR E2S UPPA CNRS 5142) Bat. IPRA,
		\endgraf
		Avenue de l'Universit\'{e}
		\endgraf
		64013 Pau
		\endgraf
		France
		\endgraf
		{\it E-mail address} {\rm jacques.giacomoni@univ-pau.fr}
	}
	
	\author[K.Sreenadh]{K.Sreenadh}
	\address{
		Konijeti Sreenadh:
		\endgraf
		Department of Mathematics
		\endgraf
		Indian Institute of Technology, Delhi, Hauz Khas
		\endgraf
		New Delhi-110016 
		\endgraf
		India
		\endgraf
		{\it E-mail address} {\rm sreenadh@maths.iitd.ac.in}
	}
	\begin{abstract}
		We consider the following doubly nonlocal nonlinear logistic problem driven by the fractional $p$-Laplacian 
		\begin{equation*}  
			\pl u = f(x,u) -\cq ~\text{in}~ \O,
			~u=0 ~\text{in}~ \Rn\setminus\O.
		\end{equation*}
		Here $ \O \subset \Rn (N\geq2)$ is a bounded domain with $ C^{1,1}$ boundary $\partial \O$, $ s \in (0,1) $, $p \in (1,\infty)$ are  such that $ps < N$. Also $p_{s,\a}^\#\leq r<\infty$
		%\leq p_{s,\a}^\ast$}
		, where  $p_{s,\a}^\#=(2N-\a)/2N$.
		% and $p_{s,\a}^*=p(N-\a/2)/N-ps$. 
		Under suitable and general assumptions on the nonlinearity $f$, we study the existence, nonexistence, uniqueness, and regularity of weak solutions. As for applications, we treat cases of subdiffusive type logistic Choquard problem. We also consider in the superdiffusive case the Brezis-Nirenberg type problem with logistic Choquard and show the existence of a nontrivial solution for a suitable choice of $\l$. Finally for a particular choice of $f$ viz. $f(x,t)=\l t^{q-1}$ with $1<p<2r<q$, we show the existence of at least one energy nodal solution.
	\end{abstract}
	\maketitle
	\section{Introduction}
	In the present article we study the following doubly nonlocal nonlinear logistic problem driven by the fractional $p$-Laplacian	\begin{equation*}  
		\p	\begin{cases}
			\pl u = f(x,u) -\cq ~&\text{in}~ \O,\\
			u=0 ~&\text{in}~ \Rn\setminus \O.
		\end{cases} 
	\end{equation*}
	Here $ \O \subset \Rn (N\geq2)$ is a bounded domain with $ C^{1,1}$ boundary $\partial \O$; $ s \in (0,1) $, $p \in (1,\infty)$ are  such that $ps < N$. Also $p_{s,\a}^\#\leq r<\infty$, where  $p_{s,\a}^\#=(2N-\a)/2N$. The leading operator is the fractional $p$-Laplacian, defined for all $ u:\Rn \rightarrow \Rr$ smooth enough and $x \in \Rn $ by 
	\begin{equation}\label{fl}
		\pl u(x)=2 \lim\limits_{\epsilon \rightarrow 0^+}\int\limits_{\{|x-y|>\epsilon\}} \frac{|u(x)-u(y)|^{p-2}(u(x)-u(y))}{|x-y|^{N+ps}} dy
	\end{equation} 	
	(which for $p=2$ reduces to the linear fractional Laplacian, up to a dimensional constant $C(N,S)>0$). The assumptions on $f$ are given below:
	\begin{itemize}
		\item [$(f_1)$] $f:\O \times \Rr\rightarrow \Rr$ is a Carath\'{e}odory function.
		\item [$(f_2)$] $f:\O \times (0,\infty)\rightarrow \Rr_+$ is measurable in $x$ and continuous in $t$.
		\item [$(f_3)$] $\displaystyle\frac{f(x,t)}{t^{p-1}}$ is decreasing.
		\item [$(f_4)$] There exists some constant $C_0>0$ such that for a.e. $x \in  \O$ and for all $t \in \Rr:$
		$$|f(x,t)\leq C_0(1+|t|^{q-1}) ,  $$ where $1<q\leq p_s^\ast$ and $p_s^\ast =\frac{Np}{N-ps}$.
	\end{itemize}
	Alternatively to $(f_4)$, we assume
	\begin{itemize} \item[$(f_5)$] $f(x,t)=\l t^{q-1}$, where $q<1$. 
		%or $q<0$ and satisfies $(sp-1)(p-q)<1$.
	\end{itemize}
	The nonlocal operators, particularly fractional $p$-Laplacian have  wide applications in the real world such as finance, obstacle problems, phase transition, image processing, material science, etc. For more details, we refer to the works \cite{C,GO,P} and the references therein. Given these real-world applications, nonlocal operators are intensively studied and there has been an ample amount of various works on problems involving such kind of operators. For the reader's convenience we first give a brief review of the literature to the problems involving Choquard and singular nonlinearities:\\
	There is a large literature available for the problems with Choquard nonlinearity due to its vast application in physical modeling see for instance the works of Pekar \cite{Pe} and Lieb \cite{Li}. For detailed studies on the existence and regularity of weak solutions for these types of problems we refer in the local setting to \cite{MS4} and the references therein. Alves, Figueiredo and Yang have considered in \cite{AFY}  a class of Schr\"odinger equations in three dimensions and via the penalization method, proved the existence of non trivial solutions. In the non local case, Choquard type equations have been investigated more recently and arise for instance in the study of mean field limit of weakly interacting molecules, in the quantum mechanical theory and in the dynamics of relativistic Boson-stars (see \cite{daveniaetal} and references therein). In \cite{daveniaetal} a Schr\"odinger type problem with Hartree type nonlinearity and involving the fractional laplacian is studied. Existence, nonexistence and properties of solutions are proved in this paper. We can also quote the references \cite{alvesetal}, \cite{Bonannoetal}, \cite{Siciliano2etal} and \cite{luxu} for semiclassical  limit type and concentration behaviour results. 
	
	For the Brezis-Niremberg type results with Choquard nonlinearity, we can refer \cite{GY} in the local setting and \cite{MS} for fractional diffusion case.\\ 
	The semi-linear problems with singular nonlinearity were studied in the seminal work of Crandall, Rabinowitz, and Tartar \cite {CRT}. This work motivates many further contributions for elliptic and parabolic singular equations where existence, uniqueness, regularity and asymptotic behaviour, multiplicity are discussed. The case  of fractional elliptic equations with singular  nonlinearities were  studied more recently by
	Barrios et al. in \cite{BBMP},  Adimurthi and et al. in \cite{adi} whereas \cite{AGW}, \cite{CMB}, and \cite{GDS} have investigated singular problems involving the $p$-fractional Laplacian. The doubly nonlocal and singular case is dealt with in \cite{jds}.\\% For the recent literature on singular problems, we  quote among the others \cite{H,HSS,HSS1} and the references therein. 
	Logistic equations found many of their applications in the real-life world. One of the most important applications is in mathematical biology where the parabolic semilinear logistic equation describes the evolution
	of spatial distribution of a biological population in the presence of constant rates of reproduction and mortality (Verhulst’s law), see \cite{GM}. This is the obvious reason why, in the study of logistic type equations, authors usually consider positive solutions. More recently, evolutive systems involving logistic terms have been studied as a model for the biological phenomenon of chemotaxis \cite{TW}. Regarding the elliptic counterpart, it models an equilibrium distribution, see \cite{CDT}. Due to these important applications, logistic equations are very broadly studied. A vast literature on logistic equations with elliptic counterpart is available, we refer for instance \cite{AB,C,IMP,MP,S} and the references therein.\\
	In \cite{PW2}, Papageorgiou and Winkert study the following nonlinear Dirichlet problem with subdiffusive and equidiffusive reactions involving Leray-Lions type operator:
	\begin{align}\label{0.2}
		-\text{div}(a(\nabla u))=\l u^{q-1}-f(x,u)~\text{in}~\O,u=0~\text{on}~\partial\O.
	\end{align}
	Here $\O \subset \Rn$ is a bounded domain with a $C^2$-boundary $\partial\O$, $1<q\leq p$ and \linebreak $f:\O \times \Rr \rightarrow \Rr$ is a Carath\'{e}odory function such that $f(x,\cdot)$ exhibits a\linebreak $(p-1)$-superlinear growth near $\pm \infty$ for a.a. $x \in \O$. By applying variational methods, the authors show that for any $\l>0$, the problem \eqref{0.2} has at least three nontrivial smooth solutions (one positive, one negative, and one sign-changing). Also in the particular case $(p,2)$-equations, they show the existence of four nontrivial smooth solutions using Morse Theory. A similar result for the superdiffusive case can be found in \cite{GRP}. Recently, in \cite{PW1} the authors consider the following singular $(p,q)$-equation with logistic perturbation
	\begin{align}
		-\Delta_p u-\Delta_q u =\l [u^{-\eta}+u^{\theta-1}]-f(x,u)~\text{in}~\O,~u=0~\text{on}~\partial\O~\text{and}~u>0~\text{in}~\O.
	\end{align}
	Here $\l>0$, $0<\eta<1$ and $1<p<q<\theta$. The function $f:\O \times \Rr \rightarrow \Rr$ is a Carath\'{e}odory function which is $(\theta-1)$-superlinear near $\infty$ for a.a. $x \in \O$. The authors prove the existence and nonexistence result for  positive solutions depending on the value of $\l$.\\
	The fractional counterpart is not so intensively studied. Recently, in \cite{IMP} Iannizzotto, Mosconi, and Papageorgiou consider the following nonlinear elliptic equation for the fractional-order
	\begin{align}
		\pl u=\l u^{q-1}-u^{r-1}~\text{in}~\O,~u>0~\text{in}~\O,~u=0~\text{in}~\Rn\setminus\O.
	\end{align}
	Here $ \O \subset \Rn (N\geq2)$ is a bounded domain with $ C^{1,1}$ boundary $\partial \O$, $ s \in (0,1) $, $p \geq 2$ are such that $ps < N$. In the subdiffusive and equidiffusive cases, the authors prove the existence and uniqueness of the positive solution for a suitable range of values of $\l$ whereas in superdiffusive case, they established a bifurcation result.\\
	As far as we know, there is no study that involves the Choquard nonlinearity as logistic reaction. We aim to fill this gap in the present paper. Using the variational methods, we show the existence and uniqueness of a positive solution between a given sub-supersolution pair. The main difficulty that arises here, is the lack of comparison principles. We overcome this difficulty by introducing a specific notion of global sub and supersolutions (inspired cooperative systems, see Section \ref{P}). Then we present some practical situations where such sub and supersolution pair can be constructed (see Corollaries \ref{c3.1} and \ref{c3.2}). We also consider the corresponding Brezis-Nirenberg type problem:
	\begin{align}\label{e1.5}
		\pl u =\l u^{p-1}+u^{p_s^\ast-1}-\cq~\text{in}~\O,~u=0~\text{in}~\Rn\setminus\O.
	\end{align}
	We show the existence of nontrivial solutions by proving a compactness result for Palais-Smale sequence at level strictly below some critical level. We make use of the auxiliary functions and estimates introduced in \cite{MPSY} to construct such Palais-Smale sequences.\\ 
	Finally, by using the minimization techniques on the Nehari nodal set, we show in the subcritical case the existence of a least energy sign changing solution to the problem
	\begin{equation*}
		(P_\l)	\begin{cases}	\pl u =\l u^{q-1}-\cq~&\text{in}~\O,\\
			u =0~&\text{in}~\Rn\setminus \O, \end{cases}
	\end{equation*} with $1<p<2r<q<p_s^\ast.$\\
	\textbf{Structure of paper:} In Section \ref{P}, we give the preliminaries required to obtain
	the main results. We also present a short result about the regularity of the nonnegative solution of $\p$. In Section \ref{RC}, we discuss the regular case and obtain the
	existence and uniqueness of the positive solution of $\p$. As examples we present the
	cases when $f(t) = \l t^{q-1}$, where $1 < q < p < r$ and $f(t) = \l t^{p-1} + \l t^{q-1}$, where again
	$1<q<p<r$. In Section \ref{BNP} we consider the Brezis-Nirenberg problem with logistic term and show the existence of solution for a suitable range of $\l$. Lastly in Section \ref{SCS}, we show the existence of at least one energy nodal solution to $(P_\l)$.\\
	\textbf{Notations:} We fix the following notations throughout the paper.
	\begin{itemize}
		\item For any $t \in \Rr$, $m >0$ we set $t^{m}=|t|^{m-1}t$.
		\item For any $D \subset \Rn$ we shall denote by $|D|$ the Lebesgue measure of $D$. 
		\item By $w_1 \leq w_2$ (resp. $w_2 \leq w_1$), we shall mean that $w_1(x)\leq w_2(x)$(resp. $w_2(x)\leq w_1(x)$) for a.e. $x \in \O$ and for any two measurable functions $w_1,w_2:\O \rightarrow \Rr$.
		\item By $w^+$ and $w^-$, we shall denote the positive and negative parts respectively of a measurable function $w$.
		% except in Section \ref{S8} where $u^-$ is defined explicitly.
		\item For any ordered set $Z$, $Z_+$ shall denote its nonnegative ordered cone.
		\item For all $p \in [1,\infty]$, $|\cdot|_{p}$ denotes the standard norm of $L^{p}(\O)$ (or $L^{p}(\Rn)$, which will be clear from the context).
		\item The positive constants are generally denoted by $C$. 	
	\end{itemize}
	\section{Preliminaries}\label{P}	
	In this section, we recall some preliminary results that are required in the later sections and also give the statements of the main results. Firstly, we collect some well known results regarding fractional Sobolev spaces. For $0<s<1$ and $1<p<\infty$, the fractional Sobolev spaces are defined as 
	$$ W^{s,p}(\Rn)= \left\{u \in L^p(\Rn): \int\limits_{\Rn}\int\limits_{\Rn} \frac{|u(x)-u(y)|^p}{|x-y|^{N+ps}}dxdy< \infty\right\}                                 $$
	equipped with the norm 
	$$\|u\|_{ W^{s,p}(\Rn)} = |u|_p + \left(     \int\limits_{\Rn}\int\limits_{\Rn} \frac{|u(x)-u(y)|^p}{|x-y|^{N+ps}}dxdy    \right)^\frac{1}{p}.$$
	
	We also define 
	$$W_{0}^{s,p}(\O) = \{ u \in W^{s,p}(\Rn):u=0~ \text{in} ~\Rn \setminus \O \} $$
	with respect to the norm
	$$\|u\| =  \left(     \int\limits_{\Rn}\int\limits_{\Rn} \frac{|u(x)-u(y)|^p}{|x-y|^{N+ps}}dxdy    \right)^\frac{1}{p} =  \left(     \int\limits_Q \frac{|u(x)-u(y)|^p}{|x-y|^{N+ps}}dxdy    \right)^\frac{1}{p}$$
	
	where $Q =\Rr^{2n}\setminus (\O^c \times \O^c)$. The space $\w$ is a uniformly convex, separable Banach space  whose dual space is denoted by $W^{-s,p^{\prime}}(\O)$ and where $p^{\prime}=\frac{p}{p-1} $ is the conjugate exponent of $p$ (see \cite{NPV}). The embedding $\w \hookrightarrow L^\nu(\O) $ is continuous for all $ \nu \in [1,p_s^\ast]$ and compact for all $\nu \in [1,p_s^\ast)$.  We also set from \cite[Definition 2.1]{IMS}, the following space:
	
	$$\widetilde{W}^{s,p}(\O)=\left\{u \in L^p_{\text{loc}}(\Rn): \exists~ U \Supset \O ~\text{s.t.} ~u \in W^{s,p}(U),~ \int\limits_{\Rn}\frac{|U(x)|^{p-1}}{(1+|x|)^{N+ps}}dx < \infty  \right    \}.   $$
	By \cite[Lemma 2.3]{IMS}, for any $u \in \w$ we can define $\pl u \in W^{-s,p^\prime}(\O)$ by setting for all $v \in \w$
	$$ \langle \pl u ,v \rangle =\int\limits_{\Rn}\int\limits_{\Rn} \frac{(u(x)-u(y))^{p-1}(v(x)-v(y))}{|x-y|^{N+ps}}dxdy .   $$
	The definition above agrees with \eqref{fl} as soon as $u$ is smooth enough (for instance, if $u \in S(\Rn)$). By \cite[Lemma~ 2.1]{FI},  $\pl : \w \rightarrow W^{-s,p^\prime}(\O)$ is a monotone, continuous, $(S)_+$-operator, namely, whenever $(u_n)$ is a
	sequence in $\w$ such that $u_n \rightharpoonup u$ in $ \w$
	and
	$$\limsup_{n} \langle \pl u_n, u_n-u\rangle \leq 0 ,$$
	then $u_n \rightarrow u$ in $\w$.\\
	Another useful property referred as $T$-monotonicity of $(-\Delta)_p^s$ is the following (see \cite[Proof of Lemma 3.2]{FI}):
	\begin{prop}\label{p2.1}
		Let $u,v \in \widetilde{W}^{s,p}(\O)$ such that $(u-v)^+ \in \w$ and satisfying $$\langle (-\Delta)_p^su-(-\Delta)_p^sv,(u-v)^+\rangle \leq 0    .    $$ Then, $u\leq v$ in $\O$. \qed
	\end{prop} 
	We will also need a strong maximum principle to ensure positivity of weak solutions. Precisely,
	\begin{prop}\label{p2.2}
		Let $g \in C^0(\Rr)\cap BV_{loc}(\Rr)$, $c:\O\rightarrow \Rr$ is a Lebesgue measurable function that satisfies $M_1 \leq c(x)\leq M_2$ a.e. $x \in \O$ for some $M_1,M_2 >0$, and  $ u \in \w \cap C^0(\overline{\O})$, $ u \not \equiv 0$ such that 
		\begin{equation*}
			\pl u +c(x)g(u)\geq c(x)g(0)~\text{weakly in }~\O,
			u \geq 0~\text{in}~\O.
		\end{equation*}
		Then, $$\inf_\O \frac{u}{d^s}>0.$$
	\end{prop}
	\begin{proof}
		Proof follows similarly as the proof of the Theorem $2.6$ in \cite{IMP}.
	\end{proof}
	%Next, we state a regularity result for the problem 
	%\begin{align}\label{rr}
	%	\pl u=f~\text{in}~\O,~u=0~\text{in}~\O^c.
	%\end{align}
	%that will help to establish the regularity result for our problem:
	%\begin{prop}\cite[Theorem 1.1]{IMS}\label{p2.3}
	%	Let $1 \leq p<\infty$ and $\O$ be a bounded domain in $\Rn$ with a $C^{1,1}$ boundary and $f \in L^\infty(\O)$. Then there exist  positive constants $C_{\O}$ and $\mu$ depending upon $N,s,p$, with $C_\O$ also depending on $\O$, such that any weak solution $u \in \w$ of \eqref{rr} satisfies 
	%	$$\|u\|_{C^\mu(\overline{\O})}     \leq C_{\O}\|f\|_{L^\infty(\O)}^\frac{1}{p-1}.                                $$
	%\end{prop}
	
	Regarding the spectral properties of the fractional p-Laplacian, we refer the reader to \cite{LL}.
	We just recall that the eigenvalue problem which is stated as
	\begin{equation}\label{evp}
		\pl u = \l u^{p-1} ~\text{in}~ \O, 
		u=0 ~\text{in} ~ \Rn\setminus\O,
	\end{equation}
	admits a principal eigenvalue defined as
	\begin{equation}\label{ev}
		\hat{\l}_1= \inf\limits_{u \in \w \setminus \{0\}} \frac{\|u\|^p}{|u|_p^p}.
	\end{equation}
	Furthermore, $\hat{\l}_1$ is isolated and simple, i.e.,  there exists a unique positive eigenfunction $\hat{u}_1 \in \text{int}(C_s^0(\overline{\O})_+) $ s.t.
	$|\hat{u}_1|_p = 1.$\\
	The crucial result to handle the nonlocal Choquard type of nonlinearity is the following well-known
	Hardy-Littlewood-Sobolev inequality.
	\begin{prop}\label{hls}
		\cite{hls}\textbf{Hardy-Littlewood-Sobolev inequality} Let $m, n>1$ and $0<\a<N$ with $1/m+1/n+\a/N=2$, $g\in L^{m}(\Rn), h\in L^n(\Rn)$. Then, there exist a sharp constant $C(m,n,N,\a)$ independent of $g$ and $h$ such  that 
		\begin{equation}\label{hlse}
			\int\limits_{\Rn}\int\limits_{\Rn}\frac{g(x)h(y)}{|x-y|^{\a}}dxdy \leq C(m,n,N,\a) |g|_m|h|_n.
		\end{equation}                             \qed
	\end{prop}
	In particular, let $g = h = |u|^r$ 
	then by Hardy-Littlewood-Sobolev inequality we get,
	$$                               \int\limits_{\Rn}\int\limits_{\Rn}\frac{|u(x)^r|u(y)|^r}{|x-y|^{\a}}dxdy$$
	is well defined if $|u|^r \in L^t(\Rn)$ with $t =\frac{2N}{2N-\alpha}>1$.
	% satisfying
	%$$ \frac{2}{t}  + \frac{\a}{N}  = 2   .         $$
	Thus, from Sobolev embedding theorems, we must have
	$$  p_{s,\a}^\# \leq r \leq  p_{s,\a}^\ast    ,               $$ where $p_{s,\a}^\ast=\displaystyle \frac{p(N-\a/2)}{N-ps}$.
	From this, for $u \in \w$ we have
	$$    \left(\int\limits_{\Rn}\int\limits_{\Rn}\frac{|u(x)|^r|u(y)|^r}{|x-y|^\a}dxdy \right)^\frac{1}{r} \leq C(r,n,\a)^\frac{1}{r}  |u|_{tr}^2 ,                                          $$
	where  $C(r,N,\a)$ is suitable constant and
	$p_{s,\a}^\# \leq r \leq p_{s,\a}^\ast$.\\
	The following result naturally follows.
	\begin{lemma} Let 
		$$[u]_r^{2r}= \int\limits_{\O}\int\limits_{\O} \frac{|u(y)|^r |u(x)|^r}{|x-y|^\a}dxdy ,$$ where $p_{s,\a}^\#\leq r \leq p_{s,\a}^\ast.$
		Then $[\cdot]_r$ defines a norm on $X_r=\{u:\O \rightarrow \Rr,[u]_r< \infty  \}$. Also $(X_r,[\cdot]_r)$ is complete.
	\end{lemma}
	\begin{proof} The proof is similar to the proof of \cite [Lemma 3.4]{MS}.
	\end{proof}
	Now, we define the notion of a weak supersolution, weak subsolution and a weak solution for the problem $\p$:
	\begin{defi}	
		Let $u \in\w$. Then
		\begin{itemize}
			\item[(a)]   $ u $ is a weak supersolution of $\p$, if $u\geq 0$ in $\Rn\setminus\O $ and for all $v \in \w_+$ 
			$$\langle \pl u ,v \rangle \geq
			\l \int\limits_{\O}f(x,u)vdx - \int\limits_{\O}\int\limits_{\O}\frac{|u(y)|^r|u(x)|^{r-2}u(x)v(x)}{|x-y|^\a}dxdy.$$
			\item[(b)]$ u $ is a weak subsolution of $\p$, if $u\leq 0$ in $\Rn\setminus\O $ and for all $v \in \w_+$ 
			$$\langle \pl u ,v \rangle \leq
			\l \int\limits_{\O}f(x,u)vdx - \int\limits_{\O}\int\limits_{\O}\frac{|u(y)|^r|u(x)|^{r-2}u(x)v(x)}{|x-y|^\a}dxdy.$$
		\end{itemize}
		We say that $u \in \w$ is a weak solution of $\p$ if it is both supersolution and subsolution of $\p$.
	\end{defi}		
	We adapt the above definition for the singular case as follows:
	\begin{defi}
		If $(f_5)$ holds and if $-q<\frac{p^2s}{sp-1}-(p-1)$, then $u \in \w$ is a weak solution of $\p$ if for any compact set $K\subset\Omega$, there exists $c_K>0$ such that $u>c_K$ in $K$ and if
		\begin{equation}\label{dws}
			\langle \pl u ,v \rangle =
			\l \int\limits_{\O}f(x,u)vdx - \int\limits_{\O}\int\limits_{\O}\frac{|u(y)|^r|u(x)|^{r-2}u(x)v(x)}{|x-y|^\a}dxdy,\end{equation}
		for all $ v \in C^\infty_c(\Omega)$.
	\end{defi}	
	\begin{remark}
		Note that if $f$ satisfies $(f_5)$ with $-q<\frac{p^2s}{sp-1}-(p-1)$, then by density arguments \eqref{dws} holds for any $v\in \w$ (see Corollary 1.2 in \cite{AGW}).
	\end{remark}
	The associated energy functional to the problem $\p$ is defined as
	\begin{equation}
		J(u)=\frac{\|u\|^p}{p}-\int\limits_\O F(x,u)dx +\frac{1}{2r}[u]_r^{2r},
	\end{equation} where $F(x,t)=\int\limits_\O f(x,\tau)d\tau.$
	Clearly if $(f_1)$ holds,  $J \in C^1(\w,\Rr)$ with the derivative given by 
	\begin{equation}
		\langle J^\prime(u),v\rangle=\langle \pl u, v\rangle -\int\limits_\O f(x,u)vdx+\int\limits_\O\int\limits_\O\frac{|u(y)|^ru(x)^{r-1}}{|x-y|^\a}dxdy.
	\end{equation}
	Clearly $u$ is a weak solution of $\p$ iff $ u$ is a critical point of $J$.\\
	Next, we define the notions of uniform subsolution and strong supersolution of problem $\p$ that are required for the existence of weak solutions.
	\begin{defi}
		We say that a function $\overline{u} \in \w$ is a strong supersolution of $\p$, if $\overline{u} \geq 0$ in $\Rn\setminus\O$ and 
		$$ \langle \pl \overline{u}, v\rangle \geq \int\limits_\O f(x,\overline{u}) vdx,$$
		for all $v \in \w_+$.
	\end{defi}
	
	\begin{defi}
		We say that a function $\underline{u} \in \w$ is a uniform subsolution relative to a strong supersolution $\bar{u}$ of $\p$, if $\underline{u} \leq 0$ in $\Rn\setminus\O$, $\underline{u}\leq \overline{u}$ and $$\langle \pl \underline{u} ,v \rangle \leq
		\l \int\limits_{\O}f(x,\underline{u})vdx - \int\limits_{\O}\int\limits_{\O}\frac{|\bar{u}(y)|^r|\underline{u}(x)|^{r-2}\underline{u}(x)v(x)}{|x-y|^\a}dxdy,$$	for all $v \in \w_+$.
	\end{defi}
	Now, we discuss the regularity of nonnegative weak solutions of the problem $\p$. In this matter, we first state a useful technical result:
	\begin{lemma}\cite[Lemma A.1]{BP}\label{3}
		Let $1<p<\infty$ and $f:\Rr \rightarrow \Rr$ be a convex function. Then
		$$|a-b|^{p-2}(a-b)[A|f^{\prime}(a)|^{p-2}f^{\prime}(a)-B|f^{\prime}(b)|^{p-2}f^{\prime}(b)]$$
		$$\geq|f(a)-f(b)|^{p-2}(f(a)-f(b))(A-B). $$
		for every $a,b \in \Rr$ and every $A,B\geq 0$.
	\end{lemma}
	We then state the following regularity result:
	\begin{theorem}\label{ape}
		Suppose $f$ satisfies $(f_4)$ or $(f_5)$, then any nonnegative weak solution $u$ of $\p$ satisfy $u \in L^\infty(\O)\cap C^\mu(\overline{\O})$, for some $\mu \in (0,1)$.
	\end{theorem}
	\begin{proof}
		We divide the proof in two cases:
		\subsection*{Case 1} Suppose $f$ satisfies $(f_4)$. Since $u$ is a weak solution of $(P_\l)$, it satisfies weakly in $\O$
		\begin{equation}\label{e2.7}
			\pl u = f(x,u)-\cq.
		\end{equation}
		We set $\phi=\xi|h^{\prime}_\epsilon(u)|^{p-2}h^{\prime}_\epsilon(u)$, where $\xi \in C_c^\infty(\O)$, $\xi>0$ and for $0<\epsilon <<1$, $$h_\epsilon(t):=(\epsilon^2+ t^2)^\frac{1}{2} .   $$
		Testing \eqref{e2.7} with $\phi$ and in addition by choosing 
		$$ a=u(x), b=u(y), A=\xi(x), ~\text{and}~B=\xi(y) ,$$ in Lemma \ref{3}, we get		
		$$\int\limits_Q\frac{|h_\epsilon(u(x))-h_\epsilon(u(y))|^{p-2}(h_\epsilon(u(x))-h_\epsilon(u(y)))(\xi(x)-\xi(y))}{|x-y|^{N+ps}}dydx $$
		\begin{align}\nonumber\label{3.2}
			\leq& \l\int\limits_\O f(x,u)|h^{\prime}_\epsilon(u(x))|^{p-1}\xi(x)dx-\int\limits_\O\int\limits_\O\frac{|u(y)|^r|u(x)|^{r-1}|h^{\prime}_\epsilon(u(x))|^{p-1}\xi(x)}{|x-y|^\a}dydx \\
			\leq& \l\int\limits_\O f(x,u)\xi(x)dx+\int\limits_\O\int\limits_\O\frac{|u(y)|^r|u(x)|^{r-1}\xi(x)}{|x-y|^\a}dydx. \end{align}
		The	inequality \eqref{3.2} holds since $u$ is nonnegative and $h^{\prime}_\epsilon\leq 1$. Since $h_\epsilon(t)$ converges to $h(t)=|t|$ as $\epsilon \rightarrow 0^+$, passing to the limit in \eqref{3.2} and using Fatou's lemma
		$$
		\int\limits_{\O}\int\limits_{\O} \frac{||u(x)|-|u(y)||^{p-2}(|u(x)|-|u(y)|)(\xi(x)-\xi(y))}{|x-y|^{N+ps}}dydx$$
		\begin{align}\label{3.3}
			\leq\int\limits_\O f(x,u)\xi(x)dx
			+\int\limits_\O\int\limits_\O\frac{|u(y)|^r|u(x)|^{r-1}\xi(x)}{|x-y|^\a}dydx,
		\end{align}
		for every $0\leq \xi \in C_c^\infty(\O).$ By density, \eqref{3.3} holds for $0\leq \xi \in \w_+$. Now, following \cite[Theorem 2.12]{BT}, we see that $u \in L^\infty(\O)$. Since $u \in L^\infty(\O)$ and $0<\a<N$, we observe that $$f(x,u)-\int\limits_\O\frac{|u(y)|^ru(x)^{r-1}}{|x-y|^\a}dy \in L^\infty(\O),$$ and so by \cite[Theorem 1.1]{MPSY}, $u \in C^\mu(\overline{\O})$ for some $\mu \in (0,s]$.
		\subsection*{Case 2} Suppose $f$ satisfies $(f_5)$. Similarly as in Case $1$, we get
		\begin{equation*}
			\int\limits_{\O}\int\limits_{\O} \frac{||u(x)|-|u(y)||^{p-2}(|u(x)|-|u(y)|)(\xi(x)-\xi(y))}{|x-y|^{N+ps}}dydx\leq \l \int\limits_\O u^{q-1}\xi(x)dx.
		\end{equation*}
		By following \cite[Theorem 6.4]{MS@}, we see that $u \in L^\infty(\O)$ and hence we can find $C_1,C_2>0$ such that $C_1d^\delta \leq u\leq C_2d^\beta$ (see Section \ref{SC} for proof and values of $\delta$ and $\beta$). Now following \cite[Theorem 1.4]{AGW}, we conclude that $u \in C^\mu(\overline{\O})$ for some $0<\mu<1$.
	\end{proof}
	The energy functional associated to the problem $(P_\l)$ is given by $$J_\l=\frac{\|u\|^p}{p}-\l\frac{|u|_q^q}{q}+\frac{[u]_r^{2r}}{2r}.$$
	We define
	$$N_\l = \{u \in \w \setminus \{0\} : \langle J_\l^{\prime}(u),u\rangle =0 \}.$$ Clearly, $N_\l$ contains all the nontrivial solutions of $(P_\l)$. Let us denote  $u^-(x)=\min\{u(x),0\}.$ Note that we are abusing the notation $u^-$ here. We will use this definition of $u^-$ only in Section \ref{SCS}, else it will denote the negative part of $u$. The sign-changing solutions of $(P_\l)$ may stay on the following set called the Nehari nodal set:
	$$M_\l =\{u \in \w: u^\pm \ne 0, \langle J_\l^{\prime}(u),u^\pm\rangle=0\}.$$
	Set
	\begin{equation}\label{e2.10}
		m_\l =\inf\limits_{u \in N_\l}J_\l(u),\end{equation}
	and 
	\begin{equation}\label{e2.11}
		c_\l =\inf\limits_{u \in N_\l}J_\l(u).
	\end{equation}
	Next,we give the statements of the main results that we obtain in this paper. The first one states the existence of a solution between an ordered  subsolution and supersolution pair:
	\begin{theorem}\label{t3.1}
		Suppose $f$ satisfies $(f_1)$ and let $\overline{u}$ be a strong supersolution and $\underline{u}$ be a uniform subsolution of $\p$ relative to $\overline{u}$ that satisfies $0\leq\underline{u}\leq \overline{u}$ a.e. in $\O$. Then there exist a nonnegative solution $u \in \w$ of $\p$ for any $r \in (1,\infty)$.
	\end{theorem}
	The second one infers the uniqueness of the postive solution under suitable conditions on $f$:
	\begin{theorem}\label{t3.2} 	Let $f$ be a function that satisfies $(f_1)$ to $(f_4)$.
		Then $\p$ has atmost one positive solution.
	\end{theorem}
	As for applications of above theorems we study the cases when $f(t)=\l t^{q-1}$ and $f(t)=\l(t^{p-1}+t^{q-1})$, where $1<q<p<r$. We have the following corollaries in this direction:
	\begin{Cor}\label{c3.1}
		Suppose $1<q<p<r$. Then for all $\l >0$ and $r \in (1,\infty)$, problem 
		\begin{equation}\label{p2.12}
			\pl u=\l u^{q-1}-\cq~\text{in}~\O, u=0~\text{in}~\Rn\setminus\O.
		\end{equation}
		admits a unique positive solution. Moreover, if $\l\rightarrow 0^+$, then $u_\l \rightarrow 0$ in $\w\cap C(\overline{\Omega})$.
	\end{Cor}
	\begin{Cor}\label{c3.2}
		Suppose $1<q<p<r$. Then for all $\l <\hat{\l}_1$ and $r \in (1,\infty)$, problem 
		\begin{equation}\label{p2.13}
			\pl u=\l (u^{p-1}+u^{q-1})-\cq~\text{in}~\O, u=0~\text{in}~\Rn\setminus\O.
		\end{equation}
		admits a unique positive solution. Moreover, if $\l\rightarrow 0^+$, then $u_\l \rightarrow 0$ in $\w\cap C(\overline{\Omega})$.
	\end{Cor}
	The next result shows the existence and uniqueness of energy solution in the singular case.
	\begin{theorem}\label{t3.3}
		Suppose $f$ satisfies $(f_5)$ with $-q<\frac{p^2s}{sp-1}-(p-1)$. Then there exist a unique solution of $\p$ for all $r\in (1,\infty)$ and for all $\l>0$.
	\end{theorem}
	The theorem below states the existence of solution to the Brezis- Nirenberg type problem:
	\begin{theorem}\label{t5.1}
		Let $1<p<p_s^\ast$, and $2r < p_s^\ast$ where $r< p_{s,\a}^\ast$, and $p,s$ are such that $sp^2 < \min\{N,(2N-\a)p-2r(N-sp)\}$. Then for all $\l<\hat{\l}_1$, there exist a nonnegative and nontrivial solution of problem \eqref{e1.5}.
	\end{theorem}
	Finally we prove existence of nodal solution in the subcritical case:
	\begin{theorem}\label{t2.11}
		Let $1< p<2r <q$, where $r < p_{s,\a}^\ast$ and let $m_\l$ and $c_\l$ be given by \eqref{e2.10} and \eqref{e2.11} respectively. Then for any $\l>0$ problem $(P_\l$) admits at least an energy sign-changing solution $u_\l \in \w$ such that $J(u_\l)=c_\l$. Moreover, $c_\l >2m_\l$.
	\end{theorem}
	\begin{remark}
		Using main results in \cite{GDS} and \cite{GDS2}, all above results can be in a straightforward way extended in case of nonhomogeneous operators as $p,q$-fractional laplacians.
	\end{remark}
	\section{Regular Case}\label{RC}
	This section is devoted to the case when $f$ is regular. We shall discuss the existence and the uniqueness of the nonnegative weak solution to the problem $\p$. We first prove the existence of the weak solution under the conditions stated in Theorem \ref{t3.1}.
	
	%\begin{theorem}\label{t3.1}
	%	Suppose $f$ satisfies $(f_1)$ and let $\overline{u}$ be a strong supersolution and $\underline{u}$ be a uniform subsolution of $\p$ relative to $\overline{u}$ that satisfies $0\leq\underline{u}\leq \overline{u}$ a.e. in $\O$. Then there exist a nonnegative solution $u \in \w$ of $\p$ for any $r \in (1,\infty)$.
	%	\end{theorem}
	\begin{proof}
		%The associated energy functional functional is given by 
		%\begin{equation}
		%	J(u)=\frac{\|u\|^p}{p}-\int\limits_\O F(x,u)dx +\frac{1}{2r}[u]_r^{2r}.
		%\end{equation}
		We choose $M >|\overline{u}|_\infty$ and define 
		\begin{equation}\label{e3.2}
			\hat{f}(x,t)=\begin{cases}
				f(x,\underline{u}(x),&\text{if}~t<\underline{u}(x),\\
				f(x,t),&\text{if}~ \underline{u}(x)\leq t \leq \overline{u}(x),\\
				f(x,\overline{u}(x)),&\text{if}~\overline{u}(x)<t.
			\end{cases}
		\end{equation}
		We also define the continuous, nonnegative and affine function $T_M$ such that \begin{equation}\label{e3.3}
			T_M(t)=\begin{cases}
				0, &\text{if}~t<0,\\
				t, &\text{if}~0\leq t\leq M,\\
				0, &\text{if}~2M\leq t.
			\end{cases}
		\end{equation}
		Now, consider the problem
		\begin{equation}\label{e3.4}
			\begin{cases}	\pl u= \hat{f}(x,u)-\displaystyle\int\limits_\O \frac{|T_M(u(y))|^r(T_M(u(x)))^{r-1}T'_M(u(x))}{|x-y|^\a}dy&\text{in}~\O,\\
				u=0 &\text{in}~\Rn\setminus\O.
			\end{cases}
		\end{equation}
		The associated energy functional is given by 
		\begin{equation*}
			\hat{J}(u)=\frac{\|u\|^p}{p}-\int\limits_\O\hat{F}(x,u)dx+\frac{[T_M(u)]_r^{2r}}{2r},
		\end{equation*}
		where $\hat{F}(x,t)=\displaystyle\int\limits_{0}^t\hat{f}(x,\tau)d\tau.$\\ Clearly $\hat{J} \in C^1(\w,\Rr).$ Also by Sobolev embedding and Hardy-Littlewood-Sobolev inequality, $\hat{J}$ is coercive and weakly lower semicontinuous. Thus there exist $u_0 \in \w$ that satisfies $$ \hat{J}(u_0)=\min_{u \in \w}\hat{J}(u).   $$ So $u_0$ satisfies \eqref{e3.4} in the weak sense. Testing \eqref{e3.4} with $ (u_0-\bar{u})^+$ and using the fact that $\bar{u}$ is a strong supersolution, we obtain
		\begin{align*}
			\langle \pl u_0, (u_0-\bar{u})^+ &= \int\limits_\O \hat{f}(x,u_0)(u_0-\bar{u})^+dx\\
			& -\int\limits_\O\int\limits_\O\frac{|T_M(u_0(y))|^r(T_M(u_0(x)))^{r-1}T'_M(u_0(x))(u_0-\bar{u})^+(x)}{|x-y|^\a}dxdy\\
			\leq & \int\limits_\O {f}(x,u_0)(u_0-\bar{u})^+dx\leq \langle \pl \bar{u}, (u_0-\bar{u})^+\rangle.
		\end{align*}	
		Using Proposition \ref{p2.1}, we conclude that $u_0\leq \bar{u}$. This also implies $u_0 <M$. Again testing 	\eqref{e3.4} with $(\underline{u}-u_0)^+$ and using the fact  that $\underline{u}$ is a uniform subsolution with respect to $\overline{u}$, we have
		\begin{align*}
			\langle \pl u_0, (\underline{u}-u_0)^+\rangle&=\int\limits_\O\hat{f}(x,u_0) (\underline{u}-u_0)^+\\
			&-\int\limits_\O\int\limits_\O\frac{|T_M(u_0(y))|^r(T_M(u_0(x)))^{r-1}T'_M(u_0(x))(\underline{u}-u_0)^+(x)}{|x-y|^\a}dxdy\\
			=& 	\int\limits_\O{f}(x,\underline{u}) (\underline{u}-u_0)^+-\int\limits_\O\int\limits_\O\frac{|u_0(y)|^r(u_0(x))^{r-1}(\underline{u}-u_0)^+(x)}{|x-y|^\a}dxdy\\
			\geq&\int\limits_\O{f}(x,\underline{u}) (\underline{u}-u_0)^+-\int\limits_\O\int\limits_\O\frac{|\overline{u}(y)|^r(\underline{u}(x))^{r-1}(\underline{u}-u_0)^+(x)}{|x-y|^\a}dxdy\\
			\geq& \langle \pl \underline{u}, (\underline{u}-u_0)^+\rangle.	
		\end{align*}
		Again using Proposition \ref{p2.1}, we see that $\underline{u}\leq u_0$. Hence, $0\leq\underline{u}\leq u_0 \leq \overline{u}< M$ and so $u_0$ is a nonnegative weak solution of $ \p$.			
	\end{proof}
	\begin{remark}
		If the conditions $\overline{u} \in L^\infty(\O)$ and $\underline{u}\geq0$ are dropped, then following the same ideas we can show the existence of weak solution for $p_{s,\a}^\# \leq r \leq p_{s,\a}^\ast$ without the cut-off argument.
	\end{remark}
	%	\begin{theorem}\label{t3.2} 	Let $f$ be a function that satisfies $(f_1)$ to $(f_4)$.
	%	Then $\p$ has atmost one positive solution.
	%	\end{theorem}
	Next we prove the uniqueness result stated in Theorem \ref{t3.2}.
	\begin{proof}
		We first show that any weak solution $u$ of $\p$ satisfies 
		\begin{equation} \label{e3.5}
			C_1 d^s \leq u \leq C_2 d^s~ \text{a.e. in}~\O,
		\end{equation}
		for some $C_1,C_2 >0$. Since $u \in L^\infty(\O)$ (see Theorem \ref{ape}), by \cite[Theorem 4.4]{IMS} we have that $u \leq C_1d^s$ a.e. in $\O$. Next, set $$c(x)=\int\limits_\O \frac{|u(y)|^r}{|x-y|^\a}dy.$$  Since $0<\a < N$ and $u \in L^\infty(\O)$, we have $$c(x)=\int\limits_\O \frac{|u(y)|^r}{|x-y|^\a}dy\leq |u|_\infty^{r}\int\limits_\O\frac{dy}{|x-y|^\a}\leq M_1.$$
		Also, $$c(x)\geq \frac{1}{(\text{diam}\O)^\a}\int\limits_\O|u(y)|^rdy \geq M_2.$$
		Hence by Proposition \ref{p2.2}, $u\geq C_1 d^s$ a.e. in $\O$ for some $C_1>0$. Now let $$V:= \{w:\O\rightarrow (0,\infty): w^\frac{1}{p}\in \w\},$$
		and define 
		$\Phi: V \rightarrow \Rr$ by $$\Phi(w)=J(w^\frac{1}{p}).$$ First, we claim that $\Phi$ is strictly convex. Indeed,
		\begin{align*}
			\Phi(w)=& \frac{1}{p}\int\limits_Q \frac{|w^\frac{1}{p}(x)-w^\frac{1}{p}(y)|^p}{|x-y|^{N+ps}}dydx -\int\limits_\O F(x,w^\frac{1}{p})dx + \frac{1}{2r}\int\limits_\O\int\limits_\O \frac{|w(y)|^\frac{r}{p}|w(x)|^\frac{r}{p}}{|x-y|^\a}dydx\\
			=& I_1(w)+I_2(w)+I_3(w),
		\end{align*}
		where $\displaystyle I_1(w)=\frac{1}{p}\int\limits_Q \frac{|w^\frac{1}{p}(x)-w^\frac{1}{p}(y)|^p}{|x-y|^{N+ps}}dydx$, $\displaystyle I_2(w)= -\int\limits_\O F(x,w^\frac{1}{p})dx$, and $\displaystyle I_3(w)=\frac{1}{2r}\int\limits_\O\int\limits_\O \frac{|w(y)|^\frac{r}{p}|w(x)|^\frac{r}{p}}{|x-y|^\a}dydx$.\\
		Since $p<r$, $I_3$ is strictly convex. Next note that 
		$$\frac{d F(x,t^\frac{1}{p})}{dt}= \frac{1}{p}f(x,t^\frac{1}{p}){t^\frac{1-p}{p}},$$ which by $(f_2)$ is decreasing. This proves that $I_2$ is convex. It remains to show that $I_1$ is convex. For this, let $w_1, w_2 \in V$. Then $w_1 =u_1^p$ and $w_2=u_2^p$, for some $u_1,u_2 \in \w$. Now, \begin{align}\nonumber\label{un}
			I_1((1-t)w_1+tw_2)=&\frac{1}{p}\int\limits_Q\frac{|[(1-t)w_1(x)+tw_2(x)]^\frac{1}{p}-[(1-t)w_1(x)+tw_2(x)]^\frac{1}{p}|^p}{|x-y|^{N+ps}}dydx\\
			=& \frac{1}{p}\int\limits_Q\frac{|[(1-t)u_1^p(x)+tu_2^p(x)]^\frac{1}{p}-[(1-t)u_1^p(y)+tu_2^p(y)]^\frac{1}{p}|^p}{|x-y|^{N+ps}}dydx.
		\end{align}
		Let $\sigma_t(x)=[(1-t)u_1^p(x)+tu_2^p(x)]^\frac{1}{p}$ for all 
		$t \in (0,1).$ Then
		by \cite[Proposition 4.1]{BF}, \begin{equation}\label{12}
			|\sigma_t(x)-\sigma_t(y)|^p \leq (1-t)|u_1(x)-u_1(y)|^p+t|u_2(x)-u_2(y)|^p. \end{equation}
		From \eqref{un} and \eqref{12}, we get
		\begin{align*}
			I_1((1-t)w_1+tw_2)\leq& \frac{(1-t)}{p}\int\limits_Q\frac{|u_1(x)-u_1(y)|^p}{|x-y|^{N+ps}}dydx+\frac{t}{p}\int\limits_Q\frac{|u_2(x)-u_2(y)|^p}{|x-y|^{N+ps}}dydx\\
			=&(1-t)I_1(w_1)+tI_1(w_2).
		\end{align*}
		This proves the claim. Now, let $u_0$ be a global minimizer of $J$. Then from \eqref{e3.5}, $v_0 =u_0^p$ is a global minimizer of $\Phi$. If possible, suppose $u_1 \in \w_+$ is another positive critical point of $J$ and let $v_1 =u_1^p \in V$. By \eqref{e3.5}, $u_1/u_0,u_0/u_1 \in L^\infty (\O)$ and so $\frac{u_1^p-u_0^p}{u_1^{p-1}}\in \w$. Now, consider $l :[0,1]\rightarrow \Rr$ defined by $$l(t)=\Phi((1-t)u_0+tu_1).$$ It is easy to see that $l$ is strictly convex and differentiable function. Note that $t=0$ is a global minimizer of $l$ and so $l^\prime(0)=0$. Then, by strict convexity of $l$, we have $l^\prime(1)>0$. But $$l^\prime(1)=\langle J^\prime(u_1),\frac{u_1^p-u_0^p}{u_1^{p-1}}\rangle=0,$$
		as $u_1$ is a critical point of $J$. Thus, we arrive at a contradiction. This proves the uniqueness.
	\end{proof}
	%	Next, we provide some examples where Theorem \ref{t3.1} and Theorem \ref{t3.2} guarantee the existence and uniqueness  nonnegative weak solution. 
	Next we give the proof of Corollaries \ref{c3.1} and \ref{c3.2} to complete this section.
	\subsection*{Proof of Corollary \ref{c3.1}} 
	%	\begin{equation}  \label{e3.8}
	%	\begin{cases}
	%		\pl u = \l |u|^{q-2}u -\cq ~&\text{in}~ \O,\\
	%		u=0 ~&\text{in}~ \Rn\setminus\O.
	%	\end{cases} 
	%	\end{equation}
	%	Here $ \O \subset \Rn (N\geq2)$ is a bounded domain with $ C^{1,1}$ boundary $\partial \O$, $ s \in (0,1) $, $p \in (1,\infty)$ are  such that $ps < N$. Also $0< \a < N$, $1<q<p_{s}^\ast$ and $p_{s,\a}^\#\leq r<\infty$. We also assume that $1 <q<p<r$. We have the following result in this case:
	%	\begin{theorem}
	%		Let $1< q< p<r$. Then for all $\l>0$ and $r \in (1,\infty)$, problem \eqref{e3.8} admits a unique positive solution $u_\l \in \w$. Moreover, if $\l\rightarrow 0^+$, then $u_\l \rightarrow 0$ in $\w\cap C(\overline{\Omega})$.
	%	\end{theorem}
	\begin{proof} We construct a strong supersolution $\bar{u}_\l$ and a uniform subsolution $\underline{u}_\l$ with respect to $\bar{u}_\l$. To this end, first consider
		\begin{align}\label{e3.9}
			\pl u = \l u^{q-1}~\text{in}~\O,
			u=0~\text{in}~\Rn\setminus\O.
		\end{align}
		We claim that this problem has a unique positive solution $\bar{v}_\l \in \w.$
		For $\l >0$ and $t \in \Rr$, we define $$g_\l(t)=\l(t^+)^{q-1},$$
		$$ G_\l(t)=\int\limits_{0}^{t} g_\l(\tau)d\tau=\frac{(t^+)^q}{q}.$$
		We set for all $u \in \w$,
		$$   E_{\l}(u) = \frac{\|u\|^p}{p} -\int\limits_\O G_\l(u)dx.$$
		Clearly, $E_\l \in C^1(\w,\Rr)$ and $$ \langle E_\l^{\prime}(u), v \rangle = \langle \pl u ,v \rangle -\l \int\limits_{\O}(u^+)^{q-1}vdx.$$	Let $\l>0$ be fixed. By Sobolev embedding theorem, it is easily seen that $E_\l$ is coercive and sequentially weakly lower semicontinuous. Thus there exists $\bar{u}_\l \in \w$ such that 
		\begin{equation}\label{e3.10}
			E_\l(\bar{u}_\l)=\inf_{\w}E_\l =:m_\l
		\end{equation}
		Now for all $ \nu >0$,
		$$  E_\l(\nu \hat{u}_1) = \nu^p\frac{\|\hat{u}_1\|^p}{p}-\l \nu^q \frac{|\hat{u}_1|_q^q}{q}. $$                                   
		Since $q<p$, $E_\l <0$ for all $\nu >0$ small enough. Hence in \eqref{e3.10} we have $m_\l <0$, implying $\bar{u}_\l \not \equiv 0$. From \eqref{e3.10}, we have $E_\l^{\prime}(\bar{u}_\l)=0$ in $W^{-s,p^{\prime}}(\O)$, i.e., we have weakly in $\O$ 
		\begin{equation}\label{e3.11}
			\pl \bar{u}_\l =\l (\bar{u}_\l^+)^{q-1} ~\text{in}~ \O.
		\end{equation}
		Testing \eqref{e3.11} with $\bar{u}_\l^- \in \w$, we get
		\begin{align*}
			\|\bar{u}_\l^-\|^p\leq& \langle\pl \bar{u}_\l,\bar{u}_\l^-\rangle
			= \int\limits_\O (\bar{u}_\l^+)^{q-1}\bar{u}_\l^- 
			=0.
		\end{align*}
		This implies $\bar{u}_\l \geq 0$. Hence \eqref{e3.11} reads as
		\begin{equation}\label{e3.12}
			\pl \bar{u}_\l =\l \bar{u}_\l^{q-1}~\text{weakly in}~ \O.
		\end{equation}
		By using \cite[Theorem 4.1]{HS}, we have $\bar{u}_\l \in L^\infty(\O)$ and then by \cite[Theorem 4.4]{IMS}, we see that $\bar{u}_\l \leq C_1d^s$, for some $C_1 >0$. Note that $\epsilon \hat{u}_1$ is a subsolution of \eqref{e3.9}, for $\epsilon$ small enough and so $\bar{u}_\l\geq \epsilon \hat{u}_1$. Thus, there exist $C_2 >0$ such that $C_2 d^s \leq \bar{u}_\l$. Hence, we have  that 
		\begin{equation}\label{e3.13}
			C_2\leq \frac{\bar{u}_\l}{d^s}\leq C_1 .
		\end{equation}
		Now, let $v_\l$ be another positive solution of \eqref{e3.9}, i.e., we have weakly in $\O$ 
		\begin{equation}\label{e3.14}
			\pl v_\l = \l v_\l^{q-1}.
		\end{equation}
		As above, $v_\l$ also satisfy \eqref{e3.13}. This imply that $\displaystyle\frac{\bar{u}_\l^p-v_\l^p}{\bar{u}_\l^{p-1}}, \displaystyle\frac{v_\l^p-\bar{u}_\l^p}{v_\l^{p-1}} \in \w$. Testing \eqref{e3.12} and \eqref{e3.14} with $\displaystyle\frac{\bar{u}_\l^p-v_\l^p}{\bar{u}_\l^{p-1}}$ and $\displaystyle\frac{v_\l^p-\bar{u}_\l^p}{v_\l^{p-1}}$ respectively and then adding, we get that
		\begin{equation}\label{e3.15}
			\langle \pl \bar{u}_\l,\frac{\bar{u}_\l^p-v_\l^p}{\bar{u}_\l^{p-1}}\rangle+	\langle \pl {u}_\l,\frac{v_\l^p-\bar{u}_\l^p}{v_\l^{p-1}}\rangle=\l \int\limits_\O(\bar{u}_\l^{q-p}-v_\l^{q-p})(\bar{u}_\l^p-v_\l^p).
		\end{equation}
		By using \cite[Lemma 1.8]{GGM} and $q<p$, we obtain from \eqref{e3.15} that $$\bar{u}_\l=v_\l~ \text{a.e. in}~ \O.$$
		Clearly, $\bar{u}_\l$ is a  strong supersolution of \eqref{p2.12}. \\
		Now, for uniform subsolution with respect to $\bar{u}_\l$ take $\underline{u}_\l=\epsilon\hat{u}_1$, where $\epsilon$ is chosen small enough such that 
		\begin{equation}\label{e3.16}
			\frac{\l(\epsilon\hat{u}_1(x))^{q-1}}{\hat{\l}_1(\epsilon\hat{u}_1(x))^{p-1}+\epsilon^{2r-1}\displaystyle\int\limits_\O\frac{|\bar{u}_\l(y)|^r|\hat{u}_1(x)|^{r-1}}{|x-y|^\a}dydx}>1,
		\end{equation}
		uniformly for  $x \in \O$. Such a choice of $\epsilon$ is possible since the limit of left-hand side of \eqref{e3.16} tends to $\infty$ as $\epsilon \rightarrow 0$ because $q<p<r$ and $\bar{u}_\l \in L^\infty(\O)$ together with $\hat{u}_1\sim d^s$. Now the existence and uniqueness of positive solution of \eqref{p2.12} is guaranteed by Theorem \ref{t3.1} and Theorem \ref{t3.2}.
		Finally, let $(\l_n)$ be a sequence in $(0,\infty)$ such that $\l_n \rightarrow 0^+$, and $u_n \in \w_+$ be the corresponding solution for all $n \in \N$, i.e, we have weakly in $\O$
		\begin{equation}\label{e3.17}
			\pl u_n =\l_n u_n^{q-1}-\left(\int\limits_\O \frac{|u_n(y)|^r}{|x-y|^\a}dy\right)(u_n(x))^{r-1}.
		\end{equation} 
		Testing \eqref{e3.17} with $u_n \in \w$, for all $n \in \N$ and using Sobolev embedding, we have
		\begin{align*}
			\|u_n\|^p = \l_n|u_n|_q^q - [u_n]_r^{2r} \leq \l_1c\|u_n\|^q.
		\end{align*}	
		
		Since $q<p$, $(u_n)$ is bounded in $\w$. By reflexivity and the compact embedding, we can pass to a subsequence such that
		$$u_n \rightharpoonup u_0~ \text{in} ~\w,$$  $$u_n \rightarrow u_0 ~\text{in}~L^\nu(\O)~\text{for all}~\nu \in [1,p_s^\ast).$$
		Also, by Theorem \ref{ape}, $u_n$ is bounded in $C^\mu(\O)$ and so upto a subsequnence $u_n \rightarrow u_0$ in $C(\overline{\O}).$	
		Testing \eqref{e3.17} with $(u_n-u_0) \in \w$, and using the facts that $0<\a<N$, $u_n \in L^{\infty}(\O)$, Sobolev embedding and H\"{o}lder inequality, we get	
		\begin{align}\label{e3.18}\nonumber
			\langle \pl u_n, u_n-u_0\rangle =& \l_n\int\limits_{\O}u_n^{q-1}(u_n-u_0)dx\\ \nonumber &-\displaystyle\int\limits_\O\int\limits_\O\frac{|u_n(y)|^r|u_n(x)|^{r-2}u_n(x)(u_n(x)-u_0(x))}{|x-y|^\a}dydx \\
			\leq&C(|u_n-u_0|_q+|u_n-u_0|_{C(\overline{\O})})\rightarrow 0.
		\end{align}
		Thus, we have				
		$$ \limsup_{n}\langle \pl u_n, u_n-u_0\rangle \leq 0.$$
		By the $(S)_+$-property of $\pl$, we have $u_n \rightarrow u_0$ in $\w$. So we can pass the limit in \eqref{e3.17} as $n \rightarrow \infty$ and get weakly in $\O$	
		$$  \pl u_0 =   -\displaystyle\left(\int\limits_\O\frac{|u_0(y)|^r}{|x-y|^\a}dy\right)|u_0(x)|^{r-2}u_0(x).                                           $$		
		Testing with $u_0 \in \w $,	we have
		$$\|u_0\|^p+[u_0]_r^{2r}= 0 ,           $$
		i.e., $u_0 =0$. Hence $u_n \rightarrow 0$ in $\w\cap C(\overline{\Omega})$ (note that $u_n$ is uniformly bounded in $C^\mu(\overline{\Omega})$). 
	\end{proof}
	\subsection*{Proof of Corollary \ref{c3.2}} %We consider the problem 
	%\begin{equation}\label{e3.19}
	%	\begin{cases}
	%	\pl u = \l u^{p-1}+\l u^{q-1}-\int\limits_\O\frac{|u(y)|^ru(x)^{r-1}}{|x-y|^\a}dy~&\text{in}~\O,\\
	%	u=0~&\text{in}~\O^c.
	%	\end{cases}
	%	\end{equation}
	%	We have the following existence and uniqueness result in the case:
	%	\begin{theorem}
	%	Let $1<q<p<r$. Then for all $\l <\hat{\l}_1$ and $r \in (1,\infty)$, problem \eqref{e3.19} admits a unique positive solution. Moreover, if $\l\rightarrow 0^+$, then $u_\l \rightarrow 0$ in $\w\cap C(\overline{\Omega})$.
	%	\end{theorem}
	
	\begin{proof}
		Proof follows similarly as the proof of Corollary \ref{c3.1} with the same uniform subsolution $\underline{u}$ and the strong supersolution $\bar{u}$ as the unique positive solution to
		\begin{align*}
			\pl u = \l (u^{q-1}+u^{p-1})~\text{in}~\O,
			u=0~\text{in}~\Rn\setminus\O.
		\end{align*}
	\end{proof}
	\section{Singular Case}\label{SC}
	In this section, we consider that $f$ satisfies $(f_5)$. Therefore the problem will look like
	\begin{equation}  \label{e4.1}
		\pl u = \l u^{q-1} -\cq ~\text{in}~ \O,
		u=0 ~\text{in}~ \Rn\setminus\O.
	\end{equation}
	We prove the existence, and uniqueness of weak solution when $-q<\frac{p^2s}{sp-1}-(p-1)$. Precisely, we prove Theorem \ref{t3.3}.
	%\begin{theorem}
	%	Suppose $f$ satisfies $(f_5)$. Then there exist a unique solution of $\p$ for all $r\in (1,\infty)$ and for all $\l>0$.
	%	\end{theorem}
	\begin{proof}
		We first consider the problem 
		\begin{equation}
			\label{e4.2}
			\pl =\l u^{q-1}~\text{in}~\O,
			u=0~\text{in}~\Rn\setminus\O.
		\end{equation}
		By the results of \cite{AGW}, problem \eqref{e4.2} has unique solution $\bar{u} \in \w $ that satisfy 
		\begin{equation}
			C_1d^\delta \leq \bar{u} \leq C_2 d^\beta,
		\end{equation}
		for some $C_1, C_2 >0$ and where $\begin{cases}
			\delta =s, \beta=s-\epsilon~&\text{if}~0\leq q<1,\\
			\delta=\beta=\frac{sp}{p-q}~&\text{if}~q<0.
		\end{cases}$\\
		Next we let $\underline{u}=\epsilon\bar{u}$, where $\epsilon>0$ is chosen small enough so that $\underline{u}$ satisfies
		\begin{equation}
			\pl \underline{u} \leq \l \underline{u}^{q-1}-\int\limits_\O\frac{\bar{u}(y)^r\underline{u}(x)^{r-1}}{|x-y|^\a}dy.
		\end{equation}
		Let $M> |\bar{u}|_\infty$ and define $\hat{f}$ and $T_M$ as in \eqref{e3.2} and \eqref{e3.3} respectively. Consider the problem
		\begin{equation}\label{e4.5}
			\pl u= \hat{f}(x,u)-\displaystyle\int\limits_\O \frac{T_M(u(y))^rT_M(u(x))^{r-1}T'_M(u(x))}{|x-y|^\a}dy~\text{in}~\O,
			u=0 ~\text{in}~\Rn\setminus\O.				
		\end{equation}
		The associated energy functional is given by 
		$$E(u)=\frac{\|u\|^p}{p}-\int\limits_\O\hat{F}(x,u)dx+\frac{[T_M(u)]_r^{2r}}{2r},$$
		where $\hat{F}(x,t)=\displaystyle\int\limits_{0}^t\hat{f}(x,\tau)d\tau.$ We first show that $E$ is G\^{a}teaux differentiable on $\w$. For this, it is sufficient to show that 
		\begin{equation}\label{e4.6}
			\lim\limits_{t \rightarrow 0}\int\limits_\O \frac{\hat{F}(x,u+tv)-\hat{F}(x,u)}{t}dx=\int\limits_\O \hat{f}(x,u)vdx,
		\end{equation}
		for all $v \in \w$. Since $\hat{f}$ is continuous in $\Omega\times \mathbb{R}^+$, $\hat{F}$ is differentiable and so by Mean Value Theorem, we have $$\frac{\hat{F}(x,u+tv)-\hat{F}(x,u)}{t}=\hat{f}(x,u+t\theta v)v\mbox{ in }\Omega$$ for some $\theta \in (0,1)$. Also $\hat{f}(x,u+t\theta v)v\leq \underline{u}^{q-1}v \in L^1(\O)$. Thus, by Lebesgue dominated convergence theorem, we have the equality \eqref{e4.6}. Now by H\"{o}lder, Hardy and Hardy-Littlewood-Sobolev inequalities, the energy functional $E$ is coercive and sequentially weakly lower semicontinuous and so there exist $u_0 \in \w$ that is a global minimizer of $E$ and satisfy weakly in $\O$
		\begin{equation*}
			\pl u_0=\l f(x,u_0)-\int\limits_\O \frac{|T_M(u_0(y))|^r(T_M(u_0(x)))^{r-1}T'_M(u_0(x)) }{|x-y|^\a}dy. 
		\end{equation*}
		Now as in Theorem \ref{t3.1}, we can show that $\underline{u}\leq u_0 \leq \bar{u}$ a.e. in $\O$ and hence $u_0$ is a weak solution of \eqref{e4.1}. Next, we show the uniqueness of the solution. For this, we claim that any weak solution $v$ of \eqref{e4.1} satisfy $\underline{v}\leq v \leq \bar{v}$, where $\bar{v}$ and $\underline{v}$ satisfy \begin{equation}\label{e4.7}
			\pl \bar{v}\geq \l \bar{v}^{q-1},
		\end{equation} 
		and \begin{equation}\label{e4.8}
			\pl \underline{v} \leq \l \underline{v}^{q-1}-\int\limits_\O \frac{v(y)^r\underline{v}^{r-1}(x)}{|x-y|^\a}dy,
		\end{equation}
		respectively weakly in $\O$. As $v$ is a weak solution, so it satisfies weakly in $\O$
		\begin{equation}
			\pl {v}= \l{v}^{q-1}-\int\limits_\O \frac{v(y)^r{v}^{r-1}(x)}{|x-y|^\a}dy.
		\end{equation}
		Since $q<1$, we have 
		\begin{align*}
			\langle \pl v-\pl \bar{v},(v-\bar{v})^+\rangle \leq  \l \int\limits_\O (v^{q-1}-\bar{v}^{q-1})(v-\bar{v})^+\leq 0.
		\end{align*}
		By Proposition \ref{p2.1}, we have $v \leq \bar{v}$. Again, since $q<1$ and $r>1$, we have
		\begin{align*}
			\langle \pl\underline{v}-\pl , \underline{v}-v \rangle \leq& \l \int\limits_\O (\underline{v}^{q-1}-v^{q-1})(\underline{v}-v)^+\\
			&- \int\limits_\O\left(\int\limits_\O\frac{|v(y)|^r}{|x-y|^\a}dy\right)(\underline{v}^{r-1}-v^{r-1})(\underline{v}-v)^+dx\\
			\leq&0,
		\end{align*}
		and so Proposition \ref{p2.1} yields $\underline{v}\leq v$. This proves our claim. Since $v \in L^\infty(\O)$, $\bar{u}$ and $\underline{u}$ satisfy \eqref{e4.7} and \eqref{e4.8} respectively, any weak solution $v$ of \eqref{e4.1} satisfy $C_1 d^\delta \leq v \leq C_2 d^\beta$, for some $C_1,C_2 >0$. The rest of the proof follows similarly as in the proof of Theorem \ref{t3.2} by using Picone type arguments.
	\end{proof}
	\section{Brezis-Niremberg for logistic Choquard}\label{BNP}
	In this section, we consider the problem \eqref{e1.5} %following Brezis-Nirenberg type problem for the logistic Choquard equation:
	%	\begin{equation}\label{e5.1}
	%	\begin{cases}
	%	\pl u = \l u^{p-1}+ u^{p_s^\ast-1}-\displaystyle\int\limits_\O\frac{|u(y)|^ru(x)^{r-1}}{|x-y|^\a}dy~&\text{in}~\O,\\
	%	u=0~&\text{in}~\O^c.
	%	\end{cases}
	%	\end{equation}
	and give the proof of Theorem \ref{t5.1}.
	%\begin{theorem}\label{t5.1}
	%	Let $1<p<p_s^\ast$, and $2r < p_s^\ast$ where $r< p_{s,\a}^\ast$, and $p,s$ are such that $sp^2 < \min\{N,(2N-\a)p-2r(N-sp)\}$. Then for all $\l<\hat{\l}_1$, there exist a solution of problem \eqref{e5.1}.
	%	\end{theorem}
	We will prove this theorem in a sequence of Lemmas.
	The associated energy functional in this case is given by 
	\begin{equation*}
		I_\l(u)=\frac{\|u\|^p}{p}-\l \frac{|u|^p}{p}-\frac{|u|^{p_s^\ast}}{p_s^\ast}+\frac{[u]_r^{2r}}{2r}.
	\end{equation*}
	Now, we define 
	\begin{equation}\label{e5.2}
		S = \inf\limits_{u \in W^{s,p}(\Rn)\setminus\{0\}}\frac{\|u\|^p}{|u|_{p_s^\ast}^p},
	\end{equation}
	which is positive by the fractional Sobolev inequality. 
	\begin{defi}
		We say that $(u_n)_{n\in\mathbb N}$ is a Palais-Smale sequence at level $c$, $(PS)_c$ for short, if $J(u_n) \rightarrow c$ in $\Rr$ and $J^\prime(u_n)\rightarrow 0$ in $W^{-s,p^\prime}(\O)$.
	\end{defi}
	\begin{lemma}\label{l5.3}
		Let $1<p<p_s^\ast$, $s \in (0,1)$, $N>sp$, and let $S$ be defined as in \eqref{e5.2}. Assume that $0$ is the only critical point of $I_\l$. Let $(u_n)_{n\in \mathbb N}$ be a $(PS)_c$, where $c \in (-\infty,\frac{s}{N}S^\frac{N}{sp})$. Then, necessarily $c=0$ and $u_n \rightarrow 0$ strongly in $\w$.
	\end{lemma}
	\begin{proof}
		Since $(u_n)_{n\in \mathbb N}$ is a $(PS)_c$, we have that 
		\begin{equation}\label{e5.3}
			I_\l(u_n)=\frac{\|u_n\|^p}{p}-\l \frac{|u_n|_p^p}{p}-\frac{|u_n|_{p_s^\ast}^{p_s^\ast}}{p_s^\ast}+\frac{[u_n]_r^{2r}}{2r}=c+o(1),
		\end{equation} and
		\begin{align}\nonumber\label{e5.4}
			\langle I_\l^\prime(u_n),v\rangle=&\langle \pl u_n,v\rangle-\l \int\limits_\O u_n^{p-1}v-\int\limits_\O u_n^{p_s^\ast-1}v\\
			&+\int\limits_\O\int\limits_\O\frac{|u_n(u)|^ru_n( x)^{r-1}v(x)}{|x-y|^\a}dxdy= o(\|v\|), \forall v \in \w
		\end{align}
		as $n \rightarrow \infty$. In particular, there exist $C>0$ such that $I_\l(u_n)\leq C$ and $\langle I_\l^\prime(u_n),v\rangle \leq C\|v\|$ for all $v \in \w$. Now, 
		\begin{align*}
			C(1+\|u_n\|)\geq &I_\l(u_n)-\frac{1}{p_s^\ast}\langle I_\l^\prime(u_n),u_n\rangle\\
			=& \left( \frac{1}{p}-\frac{1}{p_s^\ast}\right)\|u_n\|^p 
			+\l\left(\frac{1}{p_s^\ast}-\frac{1}{p}\right)|u_n|_{p}^{p} +\left(\frac{1}{2r}-\frac{1}{p_s^\ast}\right)[u_n]_r^{2r}\\
			\geq& 	\left( \frac{1}{p}-\frac{1}{p_s^\ast}\right)\left(1-\frac{\l}{\hat{\l}_1}\right)\|u_n\|^p.	
		\end{align*}
		This implies that $(u_n)_{n\in \mathbb N}$ is a bounded sequence in $\w$. By reflexivity and the compact embedding, $\w \hookrightarrow L^\nu(\O)$ for all $\nu \in [1,p_s^\ast)$, we can find a subsequence, still denoted by $(u_n)_{n\in \mathbb N}$, such that $u_n \rightharpoonup u $ weakly in $\w$, pointwisely a.e. in $\Omega$ and $ u_n \rightarrow u$ in $L^\nu(\O)$ for all $\nu \in [1,p_s^\ast)$. Noting that, $(u_n(x)-u_n(y))^{p-1}/|x-y|^{(N+sp)/p^{\prime}}$ is bounded in $L^{p^\prime}(\mathbb{R}^{2n})$ and converges to $(u(x)-u(y))^{p-1}/|x-y|^{(N+sp)/p^{\prime}}$ a.e. in $\mathbb{R}^{2n}$, and $(v(x)-v(y))/|x-y|^{(N+sp)/p} \in L^{p}(\mathbb{R}^{2n})$, we conclude that
		\begin{equation*}
			\langle \pl u_n,v\rangle \rightarrow 	\langle \pl u,v\rangle,  ~\text{for all }~v \in \w,
		\end{equation*} for a further subsequence. Similarly,
		\begin{equation*}
			\int\limits_\O u_n^{q-1}vdx \rightarrow \int\limits_\O u^{q-1}vdx, ~\text{for all }~v \in \w,~\text{for}~q=p,p_s^\ast.
		\end{equation*}
		Also, using Hardy-Littlewood-Sobolev inequality, we see that 
		\begin{equation*}
			\int\limits_\O\int\limits_\O\frac{|u_n(y)|^r|u_n(x)|^{r-1}v(x)}{|x-y|^\a}\rightarrow   	\int\limits_\O\int\limits_\O\frac{|u(y)|^r|u(x)|^{r-1}v(x)}{|x-y|^\a},                   ~\text{for all }~v \in \w.
		\end{equation*}
		So passing to the limit in \eqref{e5.4}, we see that $u$ is a weak solution of problem \eqref{e1.5} and so $u$ is a critical point of $I_\l$.
		%\begin{equation}\label{}
		%	\pl u = \l u^{p-1}+u^{p_s^\ast -1}-\int\limits_\O\frac{|u(y)|^ru(x)^{r-1}}{|x-y|^\a}dy,
		%	\end{equation}	
		Since $0$ is the only critical point of $I_\l$, we have $u=0$. Testing \eqref{e5.4} with $u_n$, we get
		\begin{align}\label{e5.6}
			\|u_n\|^p=&|u_n|_{p_s^\ast}^{p_s^\ast}+o(1)\\ \nonumber
			\leq & \frac{\|u_n\|^{p_s^\ast}}{S^\frac{p_s^\ast}{p}}+o(1),
		\end{align}
		i.e., \begin{equation}\label{e5.7}
			\|u_n\|^p(S^\frac{p_s^\ast}{p}-\|u_n\|^{p_s^\ast-p})\leq o(1).		
		\end{equation}
		Again from \eqref{e5.3}, we have
		\begin{align*}
			c=&\frac{\|u_n\|^p}{p}-\frac{|u_n|_{p_s^\ast}^{p_s^\ast}}{p_s^\ast}+o(1)\\
			=&\frac{s}{N}\|u_n\|^p +o(1)~\text{by}~\eqref{e5.6}
		\end{align*}
		i.e.,\begin{equation}\label{e5.8}
			\limsup_{n}\|u_n\|^p \leq \frac{N}{s}c <S^\frac{N}{sp}
		\end{equation}
		Combining \eqref{e5.7} and \eqref{e5.8}, we get $u_n \rightarrow 0$ in $\w$ and so $c=0$.
	\end{proof}
	Our next step is to show that $I$ has a critical level $c$ that satisfies $0<c<\frac{s}{N}S^\frac{N}{sp}$. For this, we make use of the auxiliary estimates done in \cite{MPSY}.	
	We first recall important results required. 	
	\begin{prop}\cite[Proposition 3.1]{BMS}
		Let $p\in (1,\infty), s \in (0,1), N>sp$, and let $S$ be defined as in \eqref{e5.2}. Then
		\begin{itemize}
			\item [(a)]There exist a minimizer of $S$;
			\item [(b)]For every minimizer $U$, there exist $x_0 \in \Rn$ and a constant sign monotone function $u:\Rr \rightarrow \Rr$ such that $U(x)=u(|x-x_0|)$;
			\item[(c)] Every minimizer $U$ is the weak solution of 
			$$\pl u = \l_U u^{p_s^\ast-1},$$ for some $\l_U>0$, i.e.,
			$$\int\limits_{\mathbb{R}^{2n}}\frac{(U(x)-U(y))^{p-1}(v(x)-v(y))}{|x-y|^{N+ps}}dxdy=\l_U\int\limits_\Rn U^{p_s^\ast-1}vdxdy~ \forall v \in W^{s,p}(\Rn).$$
		\end{itemize}\qed
	\end{prop}	
	Now let us fix a radially symmetric nonnegative decreasing minimizer $U=U(r)$ of $S$. Multiplying $U$ by a positive constant if necessary, we may assume that 
	\begin{equation}\label{e5.12}
		\pl U = U^{p_s^\ast-1}.
	\end{equation}
	Testing \eqref{e5.12} with $U$ and using \eqref{e5.2} shows that 
	\begin{equation}\label{e5.13}
		\|U\|^p=|U|_{p_s^\ast}^{p_s^\ast}=S^\frac{N}{sp}
	\end{equation}
	Note that , for any $\epsilon >0$, the function
	\begin{equation}\label{e5.14}
		U_\epsilon(x)=\frac{1}{\epsilon^{(N-sp)/p}}U\left(\frac{|x|}{\epsilon}\right)
	\end{equation}
	is also a minimizer of $S$ satisfying \eqref{e5.12} and \eqref{e5.13}, so after eventually rescaling we may assume that $U(0)=1$. From now on, $U$ will denote such a normalized minimizer and $U_\epsilon$ will denote the associated family of minimizers given by \eqref{e5.14}. We have the following asymptotic estimates for $U$.
	\begin{lemma}\cite[Lemma 2.2]{MPSY}\label{l5.4}
		There exist constants $C_1,C_2>0$ and $\kappa >0$ such that for all $z\geq 1$, 
		\begin{equation}
			\frac{C_1}{z^{(N-sp)/(p-1)}}\leq U(z)\leq \frac{C_2}{z^{(N-sp)/(p-1)}},
		\end{equation}
		and \begin{equation}
			\frac{U(\kappa z)}{U(z)}\leq \frac{1}{2}.
		\end{equation}\qed
	\end{lemma}
	For any $\epsilon, \delta >0$, let
	\begin{equation}
		m_{\epsilon, \delta}= \frac{U_\epsilon(\delta)}{U_\epsilon(\delta)-U_\epsilon(\kappa \delta)},
	\end{equation}
	where $\kappa$ is as in Lemma \ref{l5.4}. Let 
	\begin{equation}
		g_{\epsilon,\delta}(t)=\begin{cases}
			0,~&\text{if}~0\leq t\leq U_\epsilon(\kappa \delta)\\
			m_{\epsilon,\delta}^p(t-U_\epsilon(\kappa \delta)),~&\text{if}~U_\epsilon(\kappa\delta)\leq t\leq U_\epsilon(\delta)\\
			t+U_\epsilon(\delta)(m_{\epsilon,\delta}^{p-1}-1),~&\text{if}~t \geq U_\epsilon(\delta),
		\end{cases}
	\end{equation}
	and let \begin{equation}
		G_{\epsilon,\delta}(t)= \int\limits_{0}^tg_{\epsilon,\delta}^\prime(\tau)d\tau=\begin{cases}
			0,~&\text{if}~0\leq t\leq U_\epsilon(\kappa \delta)\\
			m_{\epsilon,\delta}(t-U_\epsilon(\kappa \delta)),~&\text{if}~U_\epsilon(\kappa\delta)\leq t\leq U_\epsilon(\delta)\\
			t,~&\text{if}~t \geq U_\epsilon(\delta).
		\end{cases}
	\end{equation}
	The functions $g_{\epsilon,\delta}$ and $G_{\epsilon,\delta}$ are nondecreasing and absolutely continuous. Consider the radially symmetric nonincreasing function $$u_{\epsilon,\delta}(z)=G_{\epsilon,\delta}(U_\epsilon(z)),$$
	which satisfies \begin{equation}\label{e5.17}
		u_{\epsilon,\delta}(z)=\begin{cases}
			U_\epsilon(z),~&\text{if}~z\leq \delta\\
			0,~&\text{if}~z\geq \kappa \delta.
		\end{cases}
	\end{equation}
	We have the following estimates for $u_{\epsilon,\delta},$
	\begin{lemma}\label{l5.6}
		There exist a constant $C>0$ such that for any $\epsilon \leq \frac{\delta}{2},$
		\begin{equation}\label{e5.21}
			\|u_{\epsilon,\delta}\|^p\leq S^\frac{N}{sp}+C\left(\frac{\epsilon}{\delta}\right)^{(N-sp)/(p-1)},
		\end{equation}
		\begin{equation}\label{e5.22}
			|u_{\epsilon,\delta}|^p \geq \begin{cases}
				\frac{1}{C} \epsilon^{sp}\log\left(\frac{\delta}{\epsilon}\right),~&\text{if}~N=sp^2\\
				\frac{1}{C}\epsilon^{sp},~&\text{if}~N>sp^2,
			\end{cases}
		\end{equation}
		\begin{equation}\label{e5.23}
			|u_{\epsilon,\delta}|_{p_s^\ast}^{p_s^\ast}\geq S^\frac{N}{sp}-C\left(\frac{\epsilon}{\delta}\right)^{N/(p-1)}
		\end{equation}
		\begin{equation}\label{e5.24}
			[u_{\epsilon,\delta}]_r^{2r}\leq C\epsilon^{2N-\alpha-2r(N-sp)/p}.
		\end{equation}
	\end{lemma}
	\begin{proof}
		Estimates \eqref{e5.21}, \eqref{e5.22}, and \eqref{e5.23} are already done in \cite[Lemma 2.7]{MPSY}. We shall establish estimate \eqref{e5.24}. By Hardy-Littlewood-Sobolev inequality, we have 
		\begin{equation}\label{e5.25}
			[u_{\epsilon,\delta}]_r^{2r}\leq C |u_{\epsilon,\delta}|_{\frac{2Nr}{2N-\a}}^{(2N-\a)/N}.
		\end{equation}
		Now, \begin{align*}
			\int\limits_\O u_{\epsilon,\delta}(z)^\frac{2Nr}{2N-\a}dz=\int\limits_{B(0,\kappa\delta)}U_\epsilon(z)^\frac{2Nr}{2N-\a}(z)dz\leq C\epsilon^{N-\frac{2Nr(N-sp)}{(2N-\a)p}},
		\end{align*}
		and so \eqref{e5.25} implies 
		\begin{equation*}
			[u_{\epsilon,\delta}]_r^{2r} \leq C \epsilon^{2N-\alpha-2r(N-sp)/p}, 
		\end{equation*}
		as required.
	\end{proof}
	In the next lemma, we show the existence of a Palais-Smale sequence at a level $c$ that satisfies $0<c<\frac{s}{N}S^\frac{N}{sp}$.
	\begin{lemma}\label{l5.7}
		Under the assumptions of Theorem \ref{t5.1}, there exist a $(PS)_c$ where $0<c<\frac{s}{N}S^\frac{N}{sp}$.
	\end{lemma}
	\begin{proof}
		We have,\begin{align*}
			I_\l(u) =& \frac{\|u\|^p}{p}-\l \frac{|u|_p^p}{p}-\frac{|u|_{p_s^\ast}^{p_s^\ast}}{p_s^\ast}+\frac{[u]_r^{2r}}{2r}\\
			\geq&\frac{1}{p}\left(1-\frac{\l}{\hat{\l}_1}\right)\|u\|^p-\frac{1}{p_s^\ast S^\frac{p_s^\ast}{p}}\|u\|^{p_s^\ast},
		\end{align*}
		so the origin is a strict local minimizer. Fix $\delta>0$ so small that $B_{\kappa \delta}(0)\subset \O$ so that supp $u_{\epsilon,\delta}\subset \O$ by \eqref{e5.17}. Noting that 
		\begin{align*}
			I_\l(\zeta u_{\epsilon,\delta})= \frac{\zeta^{p}}{p}\|u_{\epsilon,\delta}\|^p-\l \frac{\zeta^p}{p}|u_{\epsilon,\delta}|_p-\frac{\zeta^{p_s^\ast}}{p_s^\ast}{|u_{\epsilon,\delta}|_{p_s^\ast}^{p_s^\ast}}+\frac{\zeta^{2r}}{2r}[u]_r^{2r} \rightarrow -\infty~\text{as}~\zeta \rightarrow \infty,
		\end{align*}
		we fix $\zeta_0$ large enough so that $I_\l(\zeta_0 u_{\epsilon,\delta})<0$.
		Then let 
		\begin{equation*}
			\Gamma = \{\gamma \in C([0,1],\w):\gamma(0)=0,\gamma(1)=\zeta_0 u_{\epsilon,\delta}\},
		\end{equation*}
		and set
		\begin{equation*}
			c :=\inf\limits_{\gamma \in \Gamma}\max\limits_{t \in [0,1]} I_\l(\gamma(t))>0.
		\end{equation*}
		Next to show that $c<\frac{s}{N}S^\frac{N}{sp}$, it is sufficient to show that $\sup\limits_{t>0} I_\l(t u_{\epsilon,\delta})<\frac{s}{N}S^\frac{N}{sp}$. We choose $\epsilon$ small enough such that $(S^\frac{N}{sp}-C\left(\frac{\epsilon}{\delta}\right)^\frac{N}{p-1})>0$, where $C$ is obtained in Lemma \ref{l5.6}. We have,
		\begin{align*}
			I_\l(tu_{\epsilon,\delta})=&\frac{t^p}{p}\|u_{\epsilon,\delta}\|^p-\l\frac{t^p}{p}|u_{\epsilon,\delta}|_p^p-\frac{t^{p_s^\ast}}{p_s^\ast}|u_{\epsilon,\delta}|_{p_s^\ast}^{p_s^\ast}+\frac{t^{2r}}{2r}[u]_r^{2r}\\
			\leq & \frac{t^p}{p}\left[S^\frac{N}{sp}+C\left(\frac{\epsilon}{\delta}\right)^\frac{N-sp}{p-1}\right]-\frac{\l t^p}{Cp}\epsilon^{sp}-\frac{t^{p_s^\ast}}{p_s^\ast}\left[S^\frac{N}{sp}-C\left(\frac{\epsilon}{\delta}\right)^\frac{N}{p-1} \right] \\
			&+\frac{t^{2r}}{2r}C\epsilon^{2N-\a-2r(N-sp)/p}\\
			=& g_\epsilon(t)~\text{(say)}.
		\end{align*}
		Note that $g_\epsilon(0)=0$, $g_\epsilon(t)>0$ for $t$ small and $g_\epsilon(t)<0$ for $t$ large. If $t_\epsilon$ is the global maximum of $g_\epsilon$ in $(0,\infty)$, we obtain $g_\epsilon^\prime(t_\epsilon)=0$ and $g_\epsilon(t)\leq g_\epsilon(t_\epsilon)$, $\forall t\geq 0$. Also $g_\epsilon^\prime(t_\epsilon)=0$ implies $\limsup_{\epsilon\rightarrow0}t_\epsilon <\infty$. Hence,
		\begin{align*}
			g_\epsilon(t)\leq g_\epsilon(t_\epsilon) &\leq S^\frac{N}{sp}\left(\frac{t_\epsilon^p}{p}-\frac{t_\epsilon^{p_s^\ast}}{p_s^\ast}\right)+C_1\epsilon^\frac{N-sp}{p-1}+C_2\epsilon^\frac{N}{p-1}-C_3\epsilon^{sp}+C_4\epsilon^{2N-\a-2r(N-sp)/p}\\
			\leq& \sup_{t>0}S^\frac{N}{sp}\left(\frac{t^p}{p}-\frac{t^{p_s^\ast}}{p_s^\ast}\right)+C_1\epsilon^\frac{N-sp}{p-1}+C_2\epsilon^\frac{N}{p-1}-C_3\epsilon^{sp}+C_4\epsilon^{2N-\a-2r(N-sp)/p}\\
			=&\frac{s}{N}S^\frac{N}{sp}+C_1\epsilon^\frac{N-sp}{p-1}+C_2\epsilon^\frac{N}{p-1}-C_3\epsilon^{sp}+C_4\epsilon^{2N-\a-2r(N-sp)/p}\\
			<&\frac{s}{N}S^\frac{N}{sp},
		\end{align*}
		since  $sp^2 < \min\{N,(2N-\a)p-2r(N-sp)\}$ and so $C_1\epsilon^\frac{N-sp}{p-1}+C_2\epsilon^\frac{N}{p-1}-C_3\epsilon^{sp}+C_4\epsilon^{2N-\a-2r(N-sp)/p}<0$ for $\epsilon$ small enough. This proves that $c<\frac{s}{N}S^\frac{N}{sp}.$
	\end{proof}
	\subsection*{End of proof of Theorem \ref{t5.1}}
	\begin{proof}
		Combining Lemmas \ref{l5.3} and \ref{l5.7}, we see that $I_\l$ has a nontrivial critical point, which is the nontrivial weak solution of \eqref{e1.5}. Also since $I_\l(|u|)\leq I_\l(u)$ for any $u \in \w$, we see that the weak solution is nonnegative.
	\end{proof}
	\section{Existence of least energy nodal solution.}\label{SCS}
	In this section, we consider the problem $(P_\l)$. Recall the notations $J_\l$, $N_\l$, $M_\l$, $m_\l$, $c_\l$, and $u^-$ introduced in Section \ref{P}. We shall prove Theorem \ref{t2.11} in this section.
	%\begin{theorem}\label{t2.11}
	%	Let $1< p<2r <q$, where $r \leq p_{s,\a}^\ast$. Then for any $\l>0$ problem $(P_\l)$ admits a least energy sign-changing solution $u_\l \in \w$ such that $J_\l(u_\l)=c_\l$. Moreover, $c_\l >2m_\l$.
	%\end{theorem}
	We follow the techniques used in \cite{CNW} to obtain our result. First, we give some useful remarks:
	\begin{remark}\label{r8.1}
		Due to the nonlocal interactions between $u^+$ and $u^-$, we have the following inequalities:
		\begin{itemize}
			\item [(a)] $\|u\|^p \geq \|u^+\|^p +\|u^-\|^p$\\
			\item [(b)] $ [u]_r^{2r} \geq [u^+]_r^{2r}+[u^-]_r^{2r}$\\
			\item [(c)] $J_\l(u)\geq J_\l(u^+)+J_\l(u^-)$.
		\end{itemize}
	\end{remark}
	\begin{remark}\label{r8.2} {}
		For any $u \in \w$ the following inequality holds 
		\begin{align}\label{@}
			\langle \pl u^\pm, u^\pm \rangle \leq \langle \pl u, u^\pm \rangle \leq \langle \pl u, u\rangle.\end{align}
		These inequalities follow since the function $\psi(t)=|t|^{p-2}t$ satisfies the following inequalities
		$$\psi_p(c^\pm-d^\pm)(c^\pm-d^\pm)\leq \psi_p(c-d)(c^\pm-d^\pm)\leq \psi_p(c-d)(c-d),~\text{for all}~c,d\in \Rr.$$ As a result of \eqref{@}, the energy functional $J_\l$ satisfies $\langle J_\l^{\prime}(u),u^+\rangle \geq \langle J_\l^{\prime}(u^+),u^+\rangle$ and  $\langle J_\l^{\prime}(u),u^-\rangle \geq \langle J_\l^{\prime}(u^-),u^-\rangle$.
	\end{remark}
	\subsection*{Some Technical Lemmas} Now we shall give some technical lemmas which are required to prove our main result.
	\begin{lemma}\label{l7.1}
		For any $ u \in \w \setminus \{0\}$, there exists unique $ \tau_0 >0$ such that $\tau_0 u \in N_\l$.
	\end{lemma}
	\begin{proof}
		Given $u \in \w \setminus \{0\}$, define $$h_\l(t)=J_\l(tu)~\text{for all}~t\geq 0.$$
		Clarly $h_\l(0)=0$. Also $$h_\l^{\prime}(t)=\langle J_\l^{\prime}(tu),u\rangle=t^{p-1}\|u\|^p-\l t^{q-1}|u|_q^q+t^{2r-1}[u]_r^{2r},$$
		which implies that $h_\l^{\prime}(\tau_0)=0$ for some $\tau_0 >0$ iff $\tau_0 u \in N_\l.$ Since $p<2r<q$, there exists $\delta >0$ such that $h_\l(t)>0$ if $t \in (0,\delta)$ and $h_\l(t)<0$ if $t \in \left(\frac{1}{\delta},0\right)$. Since $h_\l(0)=0$, we can see that $h_\l$ has a maximum at some point $\tau_0>0$. Hence $h_\l^{\prime}(\tau_0)=0$ and so $\tau_0 u \in N_\l$. Next, note that the map $\displaystyle\frac{h^\prime(t)}{t^{q-1}}$ is strictly decreasing. This implies that $\tau_0$ is unique.
	\end{proof}
	
	\begin{lemma}\label{l7.2}
		Assume that $u \in \w$ with $u^\pm \not = 0$. Then there is a unique pair $(l_u,m_u)$ of positive numbers such that $l_u u^+ +m_u u^- \in M_\l.$
	\end{lemma}
	
	\begin{proof}
		For fixed $u \in \w$ with $u^\pm \not =0$, define $$ K(l,m)=\langle J_\l^\prime(lu^+ +mu^-),lu^+\rangle \text{ and}~   I(l,m)=\langle J_\l^\prime(lu^+ +mu^-),mu^-\rangle.$$ Also set $\O^+ =\{x \in \O : u(x)\geq 0\}$ and  $\O^- =\{x \in \O : u(x)< 0\}$. Now 
		\begin{align*}
			K(l,m)=&\langle J_\l^\prime(lu^+ +mu^-),lu^+\rangle\\ \nonumber
			>& \int\limits_Q\frac{((lu^+ + mu^-)(x)-(lu^+ + mu^-)(y))^{p-1}(lu^+(x)-lu^+(y))}{|x-y|^{N+ps}}dydx\\ 
			&- \l \int\limits_\O(lu^++mu^-)^{q-1}(lu^+)dx\\ \end{align*}%+\int\limits_\O\int\limits_\O\frac{|(lu^ +mu^-)(y)|^r((lu^+ +mu^-)(x))^{r-1}lu^+(x)}{|x-y|^{\a}}dydx\\ 
		\begin{align*}	\geq & l^p\int\limits_{\O^+}\int\limits_{\O^+}\frac{|u^+(x) -u^+(y)|^p}{|x-y|^{N+ps}}dydx +\int\limits_{\O^+}\int\limits_{\O^-}\frac{|lu^+(x)-mu^-(y)|^{p-1}(lu^+(x))}{|x-y|^{N+ps}}dydx\\ 
			&+\int\limits_{\O^-}\int\limits_{\O^+}\frac{(mu^-(x)-lu^+(y))^{p-1}(-lu^+(y))}{|x-y|^{N+ps}}dydx+l^p \int\limits_{\O^+}\int\limits_{\O^c} \frac{|u^+(x)|^p}{|x-y|^{N+ps}}dydx\\
			&+ l^p \int\limits_{\O^c}\int\limits_{\O^+} \frac{|u^+(y)|^p}{|x-y|^{N+ps}}dydx   -\l l^q\int\limits_\O |u^+|^qdx\\
			\geq&l^p\int\limits_{\O^+}\int\limits_{\O^+}\frac{|u^+(x) -u^+(y)|^p}{|x-y|^{N+ps}}dydx-\l l^q\int\limits_\O |u^+|^qdx.
		\end{align*}
		Since $p<q$, $K(l,m)>0$ for $0<l<<1$ and for any $m>0$. Similarly, we have 
		$$ I(l,m)\geq m^p \int\limits_{\O^-}\int\limits_{\O^-}   \frac{|u^-(x) -u^-(y)|^p}{|x-y|^{N+ps}}dydx-\l m^q\int\limits_\O |u^-|^qdx ,                                                $$
		and so $I(l,m)>0$ for  $0<m<<1$ and for any $l>0$.\\
		Hence by choosing $\delta_1 >0$ small, we have 
		\begin{equation}
			K(\delta_1,m)>0~,~I(l,\delta_1)>0.
		\end{equation}
		Now, for any $\delta_2 >\delta_1$, using the inequality $|a+b|^\mu<C(|a|^\mu+|b|^\mu)$, for any $a,b \in \Rr$ and $\mu>0$, we have
		\begin{align*}
			K(\delta_2,m)%=& \int\limits_Q\frac{((\delta_2u^+ + mu^-)(x)-(\delta_2u^+ + mu^-)(y))^{p-1}(\delta_2u^+(x)-\delta_2u^+(y))}{|x-y|^{N+ps}}dydx\\ 
			\leq&\delta_2^p\int\limits_{\O^+}\int\limits_{\O^+}\frac{|u^+(x) -u^+(y)|^p}{|x-y|^{N+ps}}dydx+C\delta_2^p \int\limits_{\O^+}\int\limits_{\O^-}\frac{|u^+(x)|^p}{|x-y|^{N+ps}}dydx\\ 
			&+C\delta_2m^{p-1}\int\limits_{\O^+}\int\limits_{\O^-}\frac{|u^-(x)|^{p-1}u^+(x)}{|x-y|^{N+ps}}dydx+\delta_2^p \int\limits_{\O^+}\int\limits_{\O^c}\frac{|u^+(x)|^p}{|x-y|^{N+ps}}dydx\\
			&+\delta_2^p \int\limits_{\O^c}\int\limits_{\O^+}\frac{|u^+(y)|^p}{|x-y|^{N+ps}}dydx-C\delta_2m^{p-1}\int\limits_{\O^-}\int\limits_{\O^+}\frac{|u^-(x)|^{p-2}u^-(x)u^+(y)}{|x-y|^{N+ps}}dydx\\
			&-C\delta_2^{p-1}m\int\limits_{\O^-}\int\limits_{\O^+}\frac{|u^+(y)|^{p-1}u^-(x)}{|x-y|^{N+ps}}dydx \\ &+C\delta_2^2m^{p-2}\int\limits_{\O^-}\int\limits_{\O^+}\frac{|u^-(x)|^{p-2}(u^+(y))^2}{|x-y|^{N+ps}}dydx\\
			%\\end{align*}
			%\begin{align*}
			&+\delta_2^p\int\limits_{\O^-}\int\limits_{\O^+}\frac{|u^+(y)|^p}{|x-y|{N+ps}}dydx-\l \delta_2^q\int\limits_\O |u^+|^qdx\\
			&+\delta^{2r}\int\limits_{\O^+}\int\limits_{\O^+}
			\frac{|u^+(y)|^r|u^+(x)|^r}{|x-y|^\a}dydx +\delta_2^r m^r\int\limits_{\O^+}\int\limits_{\O^-}\frac{|u^-(y)|^r|u^+(x)|^r}{|x-y|^\a}dydx.
		\end{align*}
		Since $p<2r<q$, $K(\delta_2,m)<0$ for $\delta_2$ large and $m\in [\delta_1,\delta_2]$.
		Similarly, $I(l,\delta_2)<0$ for $\delta_2$ large and $l \in [\delta_1,\delta_2]$. Hence by Miranda's theorem \cite{M}, we get a positive pair $(l_u,m_u)\in [0,\infty) \times [0,\infty)
		$ such that $l_uu^+ m_u u^- \in M_\l$. Next, we show that the pair $(l_u,m_u)$ is unique. We consider the following two cases:\\
		\textbf{Case 1:} $u \in M_\l.$\\
		In this case, we show that $(l_u,m_u)=(1,1)$ is the unique pair of numbers such that $l_u u^+ +m_u u^- \in M_\l$. For suppose $(l_0,m_0)$ be a pair of numbers such that $l_0u^+ +m_0 u^- \in M_\l$. Also without loss of generality suppose that $0<l_0 \leq m_0$. For $u \in \w$, we define
		\begin{align*}
			X^+(u)=&\int\limits_{\O^+}\int\limits_{\O^+}\frac{|u^+(x)-u^+(y)|^p}{|x-y|^{N+ps}}dydx+\int\limits_{\O^+}\int\limits_{\O^-}\frac{|u^+(x)-u^-(y)|^{p-1}u^+(x)}{|x-y|^{N+ps}}dydx\\ \nonumber
			&+\int\limits_{\O^-}\int\limits_{\O^+}\frac{|u^-(x)-u^+(y)|^{p-1}{u^+(y)}}{|x-y|^{N+ps}}dydx+\int\limits_{\O^+}\int\limits_{\O^c}\frac{|u^+(x)|^p}{|x-y|^{N+ps}}dydx\\ \nonumber
			&+\int\limits_{\O^c}\int\limits_{\O^+}\frac{|u^+(y)|^p}{|x-y|^{N+ps}}dydx.
		\end{align*}	
		\begin{align*}
			X^-(u)=&\int\limits_{\O^-}\int\limits_{\O^-}\frac{|u^-(x)-u^-(y)|^p}{|x-y|^{N+ps}}dydx+\int\limits_{\O^+}\int\limits_{\O^-}\frac{|u^+(x)-u^-(y)|^{p-1}(-u^-(y))}{|x-y|^{N+ps}}dydx\\ \nonumber
			&+\int\limits_{\O^-}\int\limits_{\O^+}\frac{|u^-(x)-u^+(y)|^{p-1}{(-u^-(x))}}{|x-y|^{N+ps}}dydx+\int\limits_{\O^-}\int\limits_{\O^c}\frac{|u^-(x)|^p}{|x-y|^{N+ps}}dydx\\ \nonumber
			&+\int\limits_{\O^c}\int\limits_{\O^-}\frac{|u^-(y)|^p}{|x-y|^{N+ps}}dydx
		\end{align*}
		$\displaystyle	Y^+(u)=\int\limits_\O\int\limits_{\O}\frac{|u(y)|^r|u^+(x)|^r}{|x-y|^\a}dydx$ and
		$\displaystyle Y^-(u)=\int\limits_\O\int\limits_{\O}\frac{|u(y)|^r|u^-(x)|^r}{|x-y|^\a}dydx$
		
		Now since $u \in M_\l$, we have $ \langle J_\l^{\prime}(u),u^\pm\rangle=0.$ This implies 
		\begin{equation}\label{7.2}
			X^+(u)=\l |u^+|_q^q -Y^+(u)
		\end{equation}
		%Now from $\langle J_\l^{\prime}(u),u^\pm\rangle=0$, we get
		%	\begin{align}\nonumber\label{7.2}
		%	\l\int\limits_\O|u^+|^qdx=&\int\limits_{\O^+}\int\limits_{\O^+}\frac{|u^+(x)-u^+(y)|^p}{|x-y|^{N+ps}}dydx+\int\limits_{\O^+}\int\limits_{\O^-}\frac{|u^+(x)-u^-(y)|^{p-1}u^+(x)}{|x-y|^{N+ps}}dydx\\ \nonumber
		%	&+\int\limits_{\O^-}\int\limits_{\O^+}\frac{|u^-(x)-u^+(y)|^{p-1}{u^+(y)}}{|x-y|^{N+ps}}dydx+\int\limits_{\O^+}\int\limits_{\O^c}\frac{|u^+(x)|^p}{|x-y|^{N+ps}}dydx\\ \nonumber
		%	&+\int\limits_{\O^c}\int\limits_{\O^+}\frac{|u^+(y)|^p}{|x-y|^{N+ps}}dydx+\int\limits_{\O^+}\int\limits_{\O^+}\frac{|u^+(x)|^r|u^+(y)|^r}{|x-y|^{\a}}dydx\\
		
		%	+&\int\limits_{\O^-}\int\limits_{\O^+}\frac{|u^+(x)|^r|u^-(y)|^r}{|x-y|^{\a}}dydx.
		%	\end{align}
		and
		\begin{equation}\label{7.3}
			X^-(u)=\l |u^-|_q^q -Y^-(u)
		\end{equation}
		%\begin{align}\nonumber\label{7.3}
		%	\l\int\limits_\O|u^-|^qdx=&\int\limits_{\O^-}\int\limits_{\O^-}\frac{|u^-(x)-u^-(y)|^p}{|x-y|^{N+ps}}dydx-\int\limits_{\O^+}\int\limits_{\O^-}\frac{|u^+(x)-u^-(y)|^{p-1}u^-(y)}{|x-y|^{N+ps}}dydx\\ \nonumber
		%	&+\int\limits_{\O^-}\int\limits_{\O^+}\frac{(u^-(x)-u^+(y))^{p-1}{u^-(x)}}{|x-y|^{N+ps}}dydx+\int\limits_{\O^-}\int\limits_{\O^c}\frac{|u^-(x)|^p}{|x-y|^{N+ps}}dydx\\ \nonumber
		%	&+\int\limits_{\O^c}\int\limits_{\O^-}\frac{|u^-(y)|^p}{|x-y|^{N+ps}}dydx+\int\limits_{\O^-}\int\limits_{\O^+}\frac{|u^-(x)|^r|u^+(y)|^r}{|x-y|^{\a}}dydx\\
		%	+&\int\limits_{\O^-}\int\limits_{\O^-}\frac{|u^-(x)|^r|u^-(y)|^r}{|x-y|^{\a}}dydx.
		%	\end{align}
		Again since $l_0 u^+ +m_0 u^- \in M_\l$, we have
		$ \langle J_\l^{\prime}(l_0u^+ + m_0 u^-), l_0u^\pm \rangle=0$. This implies
		\begin{equation}\label{7.4}
			l_0^p(X^+(u)+Z_1^+(u)+Z_2^+(u))=\l l_0^q |u^+|_q^q-l_0^rm_0^rY^+(u),
		\end{equation}
		and \begin{equation}\label{7.5}
			m_0^p(X^-(u)+Z_1^-(u)+Z_2^-(u))=\l l_0^q |u^-|_q^q-l_0^rm_0^rY^-(u),
		\end{equation}	
		where 
		\begin{small}
			$$ Z_1^+(u)=\frac{1}{l_0^{2r-p}}\int\limits_{\O^+}\int\limits_{\O^-}\frac{|u^+(x)-\frac{m_0}{l_0}u^-(y)|^{p-1}u^+(x)}{|x-y|^{N+ps}}dydx-	\int\limits_{\O^+}\int\limits_{\O^-}\frac{|u^+(x)-u^-(y)|^{p-1}u^+(x)}{|x-y|^{N+ps}}dydx,$$
			$$Z_2^+(u)=\int\limits_{\O^-}\int\limits_{\O^+}\frac{|\frac{m_0}{l_0}u^-(x)-u^+(y)|^{p-1}u^+(y)}{|x-y|^{N+ps}}dydx-\int\limits_{\O^-}\int\limits_{\O^+}\frac{|u^-(x)-u^+(y)|^{p-1}{u^+(y)}}{|x-y|^{N+ps}}dydx,$$
			$$Z_1^-(u)=\int\limits_{\O^+}\int\limits_{\O^-}\frac{|\frac{l_0}{m_0}u^+(x)-u^-(y)|^{p-1}(-u^-(y))}{|x-y|^{N+ps}}dydx -\int\limits_{\O^+}\int\limits_{\O^-}\frac{|u^+(x)-u^-(y)|^{p-1}(-u^-(y))}{|x-y|^{N+ps}}dydx,$$ 
			$$Z_2^-(u)=\int\limits_{\O^-}\int\limits_{\O^+}\frac{|u^-(x)-\frac{l_0}{m_0}u^+(y)|^{p-1}(-u^-(x))}{|x-y|^{N+ps}}dydx-\int\limits_{\O^-}\int\limits_{\O^+}\frac{|u^-(x)-u^+(y)|^{p-1}{(-u^-(x))}}{|x-y|^{N+ps}}dydx.$$	\end{small}
		%=& \int\limits_Q\frac{((l_0u^++m_0u^-)(x)-(l_0u^++m_0u^-)(y))^{p-1}(l_0u^+(x)-m_0u^+(y))}{|x-y|^{N+ps}}dydx\\
		%	&+\int\limits_\O\int\limits_\O\frac{|(l_0u^++m_0u^-)(y)|^r((l_0u^++m_0u^-)(x))^{r-1}l_0u^+(x)}{|x-y|^\a}dydx\\
		%	&-\l\int_\O(l_0u^++m_0u^-)^{q-1}(l_0u^+)dx\\
		%	=& l_0^p\int\limits_{\O^+}\int\limits_{\O^+}\frac{|u^+(x)-u^+(y)|^p}{|x-y|^{N+ps}}dydx +l_0\int\limits_{\O^+}\int\limits_{\O^-}\frac{|l_0u^+(x)-m_0u^-(y)|^{p-1}u^+(x)}{|x-y|^{N+ps}}dydx\\
		%	&+l_0 \int\limits_{\O^-}\int\limits_{\O^+}\frac{|m_0u^-(x)-l_0u^+(y)|^{p-1}u^+(y)}{|x-y|^{N+ps}}dydx+l_0^p\int\limits_{\O^+}\int\limits_{\O^c}\frac{|u^+(x)|^p}{|x-y|^{N+ps}}dydx\\ \end{align*}
		%\begin{align}\nonumber\label{7.4}
		%	&+l_0^p\int\limits_{\O^c}\int\limits_{\O^+}\frac{|u^+(y)|^p}{|x-y|^{N+ps}}dydx-\l l_0^q\int\limits_\O|u^+|^qdx\\
		%	&+l_0^{2r}\int\limits_{\O^+}\int\limits_{\O^+}\frac{|u^+(y)|^r|u^+(x)|^r}{|x-y|^\a}dydx+l_0^rm_0^r\int\limits_{\O^+}\int\limits_{\O^-}\frac{|u^-(y)|^r|u^+(x)|^r}{|x-y|^\a}dydx.
		%		\end{align}
		Since $l_0\leq m_0$, $Z_1^+(u),Z_2^+(u)>0$ and $Z_1^-(u),Z_2^-(u)<0$. Using these observations from \eqref{7.4} and \eqref{7.5}, we get 
		\begin{equation}\label{7.6}
			\displaystyle\frac{1}{l_0^{2r-p}}X^+(u)\leq \l l_0^{q-2r} |u^+|^q_q-Y^+(u)\mbox{ and }	\end{equation}
		\begin{equation}\label{7.7}
			\displaystyle\frac{1}{m_0^{2r-p}}X^-(u)\geq \l m_0^{q-2r} |u^-|^q_q-Y^-(u)	\end{equation}
		On subtracting \eqref{7.2} from \eqref{7.7}, we get 
		\begin{equation}\label{e7.8}
			\displaystyle\left(	\frac{1}{l_0^{2r-p}}-1\right) X^+(u)\leq \l (l_0^{q-2r}-1)|u^+|_q^q.
		\end{equation}
		Clearly, $l_0 <1$ contradicts \eqref{e7.8}. Hence $l_0 \geq 1$. Similarly using \eqref{7.7} and \eqref{7.3}, we get $m_0 \leq 1$. Hence we have $m_0=l_0=1$.\\
		\textbf{Case 2:} $u \not \in M_\l.$
		Suppose there exist $(l_1,m_1)$ and $(l_2,m_2)$ in $[0,\infty)\times [0,\infty)$ such that $w_1=l_1u^+ +m_1u^- \in M_\l$ and $w_2=l_2u^+ +m_2u^- \in M_\l$. Now
		$$w_2 =\left(\frac{l_2}{l_1}\right)l_1u^+ +\left(\frac{m_2}{m_1}\right)m_1u^-=\frac{l_2}{l_1}w_1^++\frac{m_2}{m_1}w_1^- \in M_\l.$$
		Since $w_1 \in M_\l$, 
		\begin{equation*}
			\frac{l_2}{l_1}=\frac{m_2}{m_1}=1.
		\end{equation*}
		Hence $l_1=l_2$ and $m_1=m_2$. This completes the proof.
	\end{proof}
	For any $u \in \w$ with $u^\pm \not=0$, define $H_u:[0,\infty+)\times[0,\infty+) \rightarrow \Rr$ by 
	\begin{equation*}
		H_u(l,m)=J_\l(lu^+ + mu^-) ~\text{for all}~l,m\geq 0.
	\end{equation*}
	\begin{lemma}\label{l7.3}
		For $u \in \w$ with $u^\pm\not=0$, $(l_u,m_u)$ is the unique maximum point of $H_u$, where $(l_u,m_u)$ is obtained in Lemma \ref{l7.2}.
	\end{lemma}
	\begin{proof}
		For $u \in \w$ with $u^\pm \not=0$,  using the inequality $|a+b|^\mu<C(|a|^\mu+|b|^\mu)$, for any $a,b \in \Rr$ and $\mu>0$, we have 
		\begin{align*}
			H_u(l,m)=& J_\l(lu^+ +mu^-)\\
			=&\frac{1}{p}\int\limits_Q\frac{|(lu^++mu^-)(x)-(lu^++mu^-)(y)|^p}{|x-y|^{N+ps}}dydx-\frac{\l}{q}\int\limits_\O|lu^+ +mu^-|^qdx\\
			&+ \frac{1}{2r}\int\limits_\O\int\limits_\O\frac{|(lu^+ +mu^-)(y)|^r|(lu^+ +mu^-)(x)|^r}{|x-y|^\a}dydx\\ \end{align*}
		%	=& \frac{l^p}{p}\int\limits_{\O^+}\int\limits_{\O^+}\frac{|u^+(x)-u^+(y)|^p}{|x-y|^{N+ps}}dydx+\frac{m^p}{p}\int\limits_{\O^-}\int\limits_{\O^-}\frac{|u^-(x)-u^-(y)|^p}{|x-y|^{N+ps}}dydx\\
		%	&+\frac{1}{p}\int\limits_{\O^+}\int\limits_{\O^-}\frac{|lu^+(x)-mu^-(y)|^p}{|x-y|^{N+ps}}dydx+\frac{1}{p}\int\limits_{\O^-}\int\limits_{\O^+}\frac{|mu^-(x)-lu^+(y)|^p}{|x-y|^{N+ps}}dydx\\
		%	&+\frac{l^p}{p}\int\limits_{\O^+}\int\limits_{\O^c}\frac{|u^+(x)|^p}{|x-y|^{N+ps}}dydx+\frac{l^p}{p}\int\limits_{\O^c}\int\limits_{\O^+}\frac{|u^+(y)|^p}{|x-y|^{N+ps}}dydx+\frac{m^p}{p}\int\limits_{\O^-}\int\limits_{\O^c}\frac{|u^-(x)|^p}{|x-y|^{N+ps}}dydx\\
		%	&+\frac{m^p}{p}\int\limits_{\O^c}\int\limits_{\O^-}\frac{|u^-(y)|^p}{|x-y|^{N+ps}}dydx-\l\frac{l^q}{q}\int\limits_\O|u^+|^qdx\\
		%	&+\frac{l^{2r}}{2r}\int\limits_{\O^+}\int\limits_{\O^+}\frac{|u^+(x)|^r|u^+(y)|^r}{|x-y|^\a}dydx+\frac{m^{2r}}{2r}\int\limits_{\O^-}\int\limits_{\O^-}\frac{|u^-(x)|^r|u^-(y)|^r}{|x-y|^\a}dydx\\
		%	&+\frac{l^{r}m^r}{2r}\int\limits_{\O^+}\int\limits_{\O^-}\frac{|u^+(x)|^r|u^-(y)|^r}{|x-y|^\a}dydx+\frac{l^{r}m^r}{2r}\int\limits_{\O^-}\int\limits_{\O^+}\frac{|u^+(y)|^r|u^-(x)|^r}{|x-y|^\a}d
		%	ydx\\
		\begin{align*}
			\leq&l^p\int\limits_{\O^+}\int\limits_{\O^+}\frac{|u^+(x)-u^+(y)|^p}{|x-y|^{N+ps}}dydx+m^p\int\limits_{\O^-}\int\limits_{\O^-}\frac{|u^-(x)-u^-(y)|^p}{|x-y|^{N+ps}}dydx\\
			&+Cl^p\int\limits_{\O^+}\int\limits_{\O^-}\frac{|u^+(x)|^p}{|x-y|^{N+ps}}dydx
			+Cm^p\int\limits_{\O^+}\int\limits_{\O^-}\frac{|u^-(y)|^p}{|x-y|^{N+ps}}dydx\\ 
			&+Cm^p\int\limits_{\O^-}\int\limits_{\O^+}\frac{|u^-(x)|^p}{|x-y|^{N+ps}}dydx
			+Cl^p\int\limits_{\O^-}\int\limits_{\O^+}\frac{|u^+(y)|^p}{|x-y|^{N+ps}}dydx\\
			&+\frac{l^p}{p}\int\limits_{\O^+}\int\limits_{\O^c}\frac{|u^+(x)|^p}{|x-y|^{N+ps}}dydx+\frac{l^p}{p}\int\limits_{\O^c}\int\limits_{\O^+}\frac{|u^+(y)|^p}{|x-y|^{N+ps}}dydx\\
			&+\frac{m^p}{p}\int\limits_{\O^-}\int\limits_{\O^c}\frac{|u^-(x)|^p}{|x-y|^{N+ps}}dydx
			+\frac{m^p}{p}\int\limits_{\O^c}\int\limits_{\O^-}\frac{|u^-(y)|^p}{|x-y|^{N+ps}}dydx-\l\frac{l^q}{q}\int\limits_\O|u^+|^qdx\\
			&-\l\frac{m^q}{q}\int\limits_\O|u^-|^qdx+\frac{l^{2r}}{2r}\int\limits_{\O^+}\int\limits_{\O^+}\frac{|u^+(x)|^r|u^+(y)|^r}{|x-y|^\a}dydx\\
			&+\frac{m^{2r}}{2r}\int\limits_{\O^-}\int\limits_{\O^-}\frac{|u^-(x)|^r|u^-(y)|^r}{|x-y|^\a}dydx
			+\frac{l^{r}m^r}{2r}\int\limits_{\O^+}\int\limits_{\O^-}\frac{|u^+(x)|^r|u^-(y)|^r}{|x-y|^\a}dydx\\
			&+\frac{l^{r}m^r}{2r}\int\limits_{\O^-}\int\limits_{\O^+}\frac{|u^+(y)|^r|u^-(x)|^r}{|x-y|^\a}dydx.
		\end{align*}
		Since $p<2r<q$, $$\lim\limits_{|(l,m)|\rightarrow \infty}H_u(l,m)\rightarrow -\infty.$$
		From the proof of \ref{l7.2}, $(l_u,m_u)$ is the unique critical point of $H_u$ in $[0,\infty)\times[0,\infty)$. Hence it is sufficient to check that a maximum point cannot be achieved on the boundary of $[0,\infty)\times[0,\infty)$. By contradiction, we suppose $(l_0,0)$ is a point of maximum for some $l_0 >0$. This means $J_\l(l_0u^+)\geq J_\l(tu^+)$ for all $t>0$. Hence, $\langle J_\l^{\prime}(l_0u^+),l_0u^+\rangle =0$, i.e.,
		\begin{equation}\label{7.8}
			\frac{1}{l_0^{2r-p}}\|u^+\|^p+[u^+]_r^{2r}=\l l_0^{q-2r}|u^+|_q^q.
		\end{equation}
		Again since $l_u u^+ +m_u u^- \in M_\l$, we have 
		\begin{align*}
			\langle J_\l^{\prime}(l_u u^+ +m_u u^-),l_u u^+\rangle=0.
		\end{align*}
		%	\begin{align}\nonumber\label{7.9}
		%	=&\int\limits_Q\frac{((l_u u^+ +m_u u^-)(x)-(l_u u^+ +m_u u^-)(y))^{p-1}(l_uu^+(x)-l_uu^+(y))}{|x-y|^{N+ps}}dydx\\ \nonumber
		%	&-\l\int\limits_\O(l_u u^+ +m_u u^-)^{q-1}lu^+dx+\int\limits_\O\int\limits_\O\frac{|(l_u u^+ +m_u u^-)(y)|^r((l_u u^+ +m_u u^-)(x))^{r-1}lu^+(x)}{|x-y|^\a}dydx\\ \nonumber
		%	=&l_u^p \int\limits_{\O^+} \int\limits_{\O^+}\frac{|u^+(x)-u^+(y)|^p}{|x-y|^{N+ps}}dydx+\int\limits_{\O^+}\int\limits_{\O^+}\frac{|l_uu^+(x)-m_uu^-(y)|^{p-1}l_uu^+(y)}{|x-y|^{N+ps}}dydx\\ \nonumber
		%	&-\int\limits_{\O^-} \int\limits_{\O^+}\frac{(m_uu^-(x)-l_uu^+(y))^{p-1}l_uu^+(y)}{|x-y|^{N+ps}}dydx+l_u^p\int\limits_{\O^+}\int\limits_{\O^c}\frac{|u^+(x)|^p}{|x-y|^{N+ps}}dydx\\ \nonumber
		%	&+l_u^p\int\limits_{\O^c}\int\limits_{\O^+}\frac{|u^+(y)|^p}{|x-y|^{N+ps}}dydx-\l l^q\int\limits_\O|u^+|^qdx
		%	+l_u^{2r}\int\limits_{\O^+}\int\limits_{\O^+}\frac{|u^+(y)|^r|u^+(x)|^r}{|x-y|^\a}dydx\\
		%	&+l_u^rm_u^r\int\limits_{\O^+}\int\limits_{\O^-}\frac{|u^-(y)|^r|u^+(x)|^r}{|x-y|^\a}dydx
		%	\end{align}
		Now
		%	\begin{align*}
		%	\langle J_\l^{\prime}(l_uu^+),l_uu^+\rangle=&\int\limits_Q\frac{|l_uu^+(x)-l_uu^-(y)|^p}{|x-y|^{N+ps}}dydx-l_u^q\int\limits_\O|u^+|^qdx+l_u^{2r}\int\limits_{\O^+}\int\limits_{\O^+}\frac{|u^+(y)|^r|u^+(x)|^r}{|x-y|^\a}dydx\\
		%	=&l_u^p\int\limits_{\O^+}\int\limits_{\O^+}\frac{|u^+(x)-u^+(y)|^p}{|x-y^{N+ps}}dydx+\int\limits_{\O^-}\int\limits_{\O^+}\frac{|l_uu^+(y)|^p}{|x-y|^{N+ps}}dydx+\int\limits_{\O^+}\int\limits_{\O^-}\frac{|l_uu^+(x)|^p}{|x-y|^{N+ps}}dydx\\
		%	&+l_u^p\int\limits_{\O^+}\int\limits_{\O^c}\frac{|u^+(x)|^p}{|x-y|^{N+ps}}dydx+l_u^p\int\limits_{\O^c}\int\limits_{\O^+}\frac{|u^+(y)|^p}{|x-y|^{N+ps}}dydx-\l l_u^q\int\limits_\O|u^+|^qdx\\
		%	&+l_u^{2r}\int\limits_{\O^+}\int\limits_{\O^+}\frac{|u^+(y)|^r|u^+(x)|^r}{|x-y|^\a}dydx\\
		%	\leq&l_u^p \int\limits_{\O^+} \int\limits_{\O^+}\frac{|u^+(x)-u^+(y)|^p}{|x-y|^{N+ps}}dydx+\int\limits_{\O^+} \int\limits_{\O^+}\frac{|l_uu^+(x)-m_uu^-(y)|^{p-1}l_uu^+(y)}{|x-y|^{N+ps}}dydx\\
		%	&-\int\limits_{\O^-} \int\limits_{\O^+}\frac{(m_uu^-(x)-l_uu^+(y))^{p-1}l_uu^+(y)}{|x-y|^{N+ps}}dydx+l_u^p\int\limits_{\O^+}\int\limits_{\O^c}\frac{|u^+(x)|^p}{|x-y|^{N+ps}}dydx\\
		%	&+l_u^p\int\limits_{\O^c}\int\limits_{\O^+}\frac{|u^+(y)|^p}{|x-y|^{N+ps}}dydx-\l l_u^q\int\limits_\O|u^+|^qdx+l_u^{2r}\int\limits_{\O^+}\int\limits_{\O^+}\frac{|u^+(y)|^r|u^+(x)|^r}{|x-y|^\a}dydx\\
		%	&+l_u^rm_u^r\int\limits_{\O^+}\int\limits_{\O^-}\frac{|u^-(y)|^r|u^+(x)|^r}{|x-y|^\a}dydx.
		%	\end{align*}
		\begin{align*}
			\langle J_\l^{\prime}(l_uu^+),l_uu^+\rangle\leq 	\langle J_\l^{\prime}(l_uu^++m_u u^-),l_uu^+\rangle\leq0
		\end{align*} 	 i.e.,
		\begin{equation}\label{7.10}
			\frac{1}{l_u^{2r-p}}\|u^+\|^p+[u^+]_r^{2r}\leq\l l_u^{q-2r}|u^+|_q^q.
		\end{equation}
		On subtracting \eqref{7.8} from \eqref{7.10}, we get
		\begin{equation}\label{7.11}
			\left(\frac{1}{l_u^{2r-p}}-\frac{1}{l_0^{2r-p}}\right)\|u^+\|^p\ \leq \l(l_u^{q-2r}-l_0^{q-2r})|u^+|_q^q.
		\end{equation}
		If $l_0>l_u$, we get a contradiction to \eqref{7.11}. Hence we have $l_0\leq l_u$. Now 
		\begin{align*}
			J_\l(l_0,0)=& J_\l(l_0u^+)\\
			=&  J_\l(l_0u^+)-\frac{1}{2r}\langle J_\l^{\prime}(l_0u^+),l_0u^+\rangle\\ 
			=&\left(\frac{1}{p}-\frac{1}{2r}\right)l_0^p\|u^+\|^p+\l\left(\frac{1}{2r}-\frac{1}{q}\right)l_0^q|u^+|_q^q\\
			\leq&\left(\frac{1}{p}-\frac{1}{2r}\right)\|l_u u^+\|^p+\l \left(\frac{1}{2r}-\frac{1}{q}\right)|l_u u^+|_q^q\\
			<&\left(\frac{1}{p}-\frac{1}{2r}\right)\|l_u u^+ +m_u u^-\|^p+\l\left(\frac{1}{2r}-\frac{1}{q}\right)|l_u u^++m_u u^-|_q^q\\
			=&J_\l(l_u u^++m_u u^-)-\frac{1}{2r}\langle J_\l^{\prime}(l_u u^+ +m_u u^-),l_u u^+ + m_u u^-\rangle\\
			=&J_\l(l_u u^++m_u u^-).
		\end{align*}
		This yields the contradiction. Similarly $H_u$ cannot achieve its global maximum on $(0,m)$ with $m>0$.\\
		This completes the proof.
	\end{proof}
	\begin{lemma}\label{l7.4}
		For any $u \in M_\l$, there exists $k_1,k_2 >0$ such that
		\begin{enumerate}
			\item $\|u^\pm\|\geq k_1$.
			\item $|u^\pm|_q^q\geq k_2$.
		\end{enumerate}
	\end{lemma}
	\begin{proof}
		Since $ u \in M_\l$, $\langle J_\l^{\prime}(u),u^+\rangle=0$. Then as in Lemma \ref{l7.3} $\langle J_\l^{\prime}(u^+),u^+\rangle <0$. This implies using Sobolev embedding that
		\begin{equation}\label{7.12}
			\|u^+\|^p<\l|u^+|_q^q -[u^+]_r^{2r}\leq C\|u^+\|^q.
		\end{equation}
		Since $p<q$, $\|u^+\|>k_1$. Also from \eqref{7.12}, we have $|u^+|_q^q \geq k_2$. Similarly, we have $\|u^-\| \geq k_1$ and $|u^-|_q^q\geq k_2$.
	\end{proof}
	\begin{lemma}\label{l7.5}
		$J_\l$ is bounded from below on $M_\l$.
	\end{lemma}
	\begin{proof}
		%Let $u \in M_\l$. Then obviously $\langle J_\l^{\prime}(u),u\rangle =0$. Now 
		%\begin{align*}
		%	J_\l(u)=&J_\l(u)-\frac{1}{2r}\langle J_\l^{\prime}(u),u\rangle \\
		%	=& \left(\frac{1}{p}-\frac{1}{2r}\right)\|u\|^p+\l \left(\frac{1}{2r}-\frac{1}{q}\right)|u|_q^q\\
		%	>& 0.
		%\end{align*}
		%This completes the proof. 
		Proof is standard and thus omitted.
	\end{proof}
	Hence $c_\l =\inf\limits_{u\in M_\l}J_\l(u)$ is well defined. Next, we show that $c_\l$ is achieved on $M_\l$.
	\begin{lemma}\label{l7.6}
		There exists $u_\l \in M_\l$ such that $J_\l(u_\l)=c_\l$.
	\end{lemma}
	\begin{proof}
		Let $(u_n)_{n\in \mathbb N}\subset M_\l$ be a minimizing sequence i.e., $J_\l(u_n)\rightarrow c_\l$ as $ n \rightarrow \infty$.
		In particular, $J_\l(u_n)$ is bounded and so we can find $C>0$ such that $J_\l(u_n)\leq C$ for all $n$. Now as in Lemma \ref{l7.5}, we see that $(u_n)_{n\in \mathbb N}$ is uniformly bounded in $\w$. Then by reflexivity and the compact embedding there exist $u_\l \in \w$ such that 
		\begin{equation}
			u_n^\pm \rightharpoonup u_\l^\pm~\text{in}~\w,
		\end{equation}
		\begin{equation}\label{8.14}
			u_n^\pm \rightarrow u_\l^\pm~\text{in}~L^\nu(\O),~\text{for all}~\nu \in [1,p_s^\ast),
		\end{equation}
		\begin{equation}\label{8.15}
			u_n^\pm(x)\rightarrow u_\l
			^\pm(x)~\text{a.e}~x \in \O.
		\end{equation}
		Clearly, by Lemma \ref{l7.4}, $u_\l^\pm \not=0$. So by Lemma \ref{l7.2} there exist $(l_\l,m_\l)\in [0,\infty)\times [0,\infty)$ such that $l_\l u_\l^+ +m_\l u_\l^- \in M_\l$. We claim that $l_\l,m_\l \leq 1$. Without loss of generality assume that $l_\l\leq m_\l$. Then from $\langle J_\l^{\prime}(l_\l u_\l^+ +m_\l u_\l^-),m_\l u_\l^-\rangle=0$ and $l_\l\leq m_\l$, we have as in Lemma \ref{7.2}
		\begin{equation}\label{7.16}
			\frac{1}{m_\l^{2r-p}} X^-(u_\l)\geq \l m_\l^{q-2r}|u_\l^-|_q^q-Y^-(u_\l).
		\end{equation}
		%\begin{align}\nonumber\label{7.16}
		%	\l m_\l^q |u_\l^-|_q^q \leq&  m_\l^p \int\limits_{\O^-}\int\limits_{\O^-}\frac{|u_\l^-(x)-u_\l^-(y)|^p}{|x-y|^{N+ps}}dydx+m_\l^p \int\limits_{\O^-}\int\limits_{\O^+}\frac{(u_\l^-(x)-\frac{l_\l}{m_\l}u_\l^+(y))^{p-1}u_\l^-(x)}{|x-y|^{N+ps}}dydx\\ \nonumber
		%	&+m_\l^p\int\limits_{\O^+}\int\limits_{\O^-}\frac{|\frac{l_\l}{m_\l}u_\l^+(x)-u_\l^-(y)}{|x-y|^{N+ps}}dydx+m_\l^{2r}\int\limits_{\O^-}\int\limits_{\O^-}\frac{|u_\l^-(y)|^r|u_\l^-(x)|^r}{|x-y|^\a}dydx\\ \nonumber
		%	&+m_\l^p\int\limits_{\O^-}\int\limits_{\O^c}\frac{|u^-_\l(x)|^p}{|x-y|^{N+ps}}dydx+m_\l^p\int\limits_{\O^c}\int\limits_{\O^-}\frac{|u^-_\l(y)|^p}{|x-y|^{N+ps}}dydx\\
		%	&+m_\l^{2r}\int\limits_{\O^-}\int\limits_{\O^+}\frac{|u_\l^+(y)|^r|u_\l^-(x)|^r}{|x-y|^\a}dydx
		%	\end{align}
		Again since $u_n \in M_\l$, we have $\langle J_\l^{\prime}(u_n),u_n^-\rangle=0$. This implies 
		\begin{equation}\label{7.17}
			X^-(u_n)=\l|u_n^-|^q_q-Y^-(u_n).
		\end{equation}
		%	\begin{align}\nonumber\label{7.17}
		%	\l|u_n^-|_q^qdx=& \int\limits_{\O^-}\int\limits_{\O^-}\frac{|u_n^-(x)-u_n^-(y)|^p}{|x-y|^{N+ps}}dydx+ \int\limits_{\O^-}\int\limits_{\O^+}\frac{(u_n^-(x)-u_n^+(y))^{p-1}u_n^-(x)}{|x-y|^{N+ps}}dydx\\ \nonumber
		%	&+\int\limits_{\O^+}\int\limits_{\O^-}\frac{|u_n^+(x)-u_n^-(y)|^{p-1}(-u_n^-(y))}{|x-y|^{N+ps}}dydx+\int\limits_{\O^-}\int\limits_{\O^-}\frac{|u_n^-(y)|^r|u_n^-(x)|^r}{|x-y|^\a}dydx\\\nonumber
		%	&\int\limits_{\O^-}\int\limits_{\O^c}\frac{|u^-_n(x)|^p}{|x-y|^{N+ps}}dydx+\int\limits_{\O^c}\int\limits_{\O^-}\frac{|u^-_n(y)|^p}{|x-y|^{N+ps}}dydx\\
		%	&+\int\limits_{\O^-}\int\limits_{\O^+}\frac{|u_n^+(y)|^r|u_n^-(x)|^r}{|x-y|^\a}dydx
		%	\end{align}
		Taking limit $n \rightarrow \infty$ in \eqref{7.17} and using \eqref{8.14}, \eqref{8.15} and Fatou's Lemma, we obtain
		\begin{equation}\label{7.18}
			X^-(u_\l)\leq \l|u_\l^-|_q^q-Y^-(u_\l).
		\end{equation}
		%	\begin{align}\nonumber\label{7.18}
		%	\l|u_\l^-|_q^qdx\geq& \int\limits_{\O^-}\int\limits_{\O^-}\frac{|u_\l^-(x)-u_\l^-(y)|^p}{|x-y|^{N+ps}}dydx+ \int\limits_{\O^-}\int\limits_{\O^+}\frac{(u_\l^-(x)-u_\l^+(y))^{p-1}u_\l^-(x)}{|x-y|^{N+ps}}dydx\\ \nonumber
		%	&+\int\limits_{\O^+}\int\limits_{\O^-}\frac{|u_\l^+(x)-u_\l^-(y)|^{p-1}(-u_\l^-(y))}{|x-y|^{N+ps}}dydx+\int\limits_{\O^-}\int\limits_{\O^-}\frac{|u_\l^-(y)|^r|u_\l^-(x)|^r}{|x-y|^\a}dydx\\ \nonumber
		%	&+\int\limits_{\O^-}\int\limits_{\O^c}\frac{|u^-_\l(x)|^p}{|x-y|^{N+ps}}dydx+\int\limits_{\O^c}\int\limits_{\O^-}\frac{|u^-_\l(y)|^p}{|x-y|^{N+ps}}dydx\\
		%	&+\int\limits_{\O^-}\int\limits_{\O^+}\frac{|u_\l^+(y)|^r|u_\l^-(x)|^r}{|x-y|^\a}dydx
		%	\end{align}
		On subtracting \eqref{7.18} from \eqref{7.16}, we obtain
		\begin{equation}\label{7.19}
			\left(\frac{1}{m_\l^{2r-p}}-1\right)X^-(u_\l)\geq \l (m_\l^{q-2r}-1)|u_\l^-|_q^q.
		\end{equation}
		%	\begin{align*}
		%	\l(1-{m_\l^{q-2r}})|u_\l^-|_q^qdx\geq&\left(1-\frac{1}{m_\l^{2r-p}}\right)\int\limits_{\O^-}\int\limits_{\O^-}\frac{|u_\l^-(x)-u_\l^-(y)|^p}{|x-y|^{N+ps}}dydx\\ \nonumber
		%	&+\int\limits_{\O^-}\int\limits_{\O^+}\frac{(u_\l^-(x)-u_\l^+(y))^{p-1}u_\l^-(x)}{|x-y|^{N+ps}}dydx\\ \nonumber
		%	&-\frac{1}{m_\l^{2r-p}}\int\limits_{\O^-}\int\limits_{\O^+}\frac{(u_\l^-(x)-\frac{l_\l}{m_\l}u_\l^+(y))^{p-1}u_\l^-(x)}{|x-y|^{N+ps}}dydx\\ \nonumber
		%	&+\int\limits_{\O^+}\int\limits_{\O^-}\frac{|u_\l^+(x)-u_\l^-(y)|^{p-1}(-u_\l^-(y))}{|x-y|^{N+ps}}dydx\\
		%	&-\frac{1}{m_\l^{2r-p}}\int\limits_{\O^+}\int\limits_{\O^-}\frac{|\frac{l_\l}{m_\l}u_\l^+(x)-u_\l^-(y)|^{p-1}(-u_\l^-(y))}{|x-y|^{N+ps}}dydx.
		%	\end{align*}
		%	\begin{align}\label{7.19}
		%	+\left(1-\frac{1}{m_\l^{2r-p}}\right)\int\limits_{\O^-}\int\limits_{\O^c}\frac{|u^-_\l(x)|^p}{|x-y|^{N+ps}}dydx+\left(1-\frac{1}{m_\l^{2r-p}}\right)\int\limits_{\O^c}\int\limits_{\O^-}\frac{|u^-_\l(y)|^p}{|x-y|^{N+ps}}dydx.
		%	\end{align}
		Clearly, $m_\l>1$ yields a contradiction to \eqref{7.19}. Hence we have $l_\l\leq m_\l \leq 1$. Now,
		\begin{align}\nonumber\label{7.20}
			c_\l\leq &J_\l(l_\l u_\l^+ +m_\l u_\l^-)\\ \nonumber
			=& J_\l(l_\l u_\l^+ +m_\l u_\l^-)-\frac{1}{2r}\langle J_\l^{\prime}(l_\l u_\l^+ +m_\l u_\l^-),l_\l u_\l^+ +m_\l u_\l^-\rangle\\ \nonumber
			=& \left(\frac{1}{p}-\frac{1}{2r}\right) \|l_\l u_\l^+ +m_\l u_\l^-\|^p + \l \left(\frac{1}{2r}-\frac{1}{q}\right)|l_\l u_\l^+ +m_\l u_\l^-|_q^q\\ \nonumber
			=&\left(\frac{1}{p}-\frac{1}{2r}\right)\left[l_\l^p\int\limits_{\O^+}\int\limits_{\O^+}\frac{|u_\l^+(x)-u_\l^+(y)|^p}{|x-y|^{N+ps}}dydx+ \int\limits_{\O^-}\int\limits_{\O^-}\frac{|u_\l^-(x)-u_\l^-(y)|^p}{|x-y|^{N+ps}}dydx\right]\\ \nonumber
			+&\left(\frac{1}{p}-\frac{1}{2r}\right)\left[\int\limits_{\O^+}\int\limits_{\O^-}\frac{|l_\l u_\l^+(x)-m_\l u_\l^-(y)|^p}{|x-y|^{N+ps}}dydx+\int\limits_{\O^-}\int\limits_{\O^+}\frac{|l_\l u_\l^+(y)-m_\l u_\l^-(x)|^p}{|x-y|^{N+ps}}dydx\right]\\ \nonumber
			+&\left(\frac{1}{p}-\frac{1}{2r}\right)\left[l_\l^p\int\limits_{\O^+}\int\limits_{\O^c}\frac{|u_\l^+(x)|^p}{|x-y|^{N+ps}}dydx+l_\l^p\int\limits_{\O^c}\int\limits_{\O^+}\frac{|u_\l^+(y)|^p}{|x-y|^{N+ps}}dydx\right]\\ \nonumber	+&\left(\frac{1}{p}-\frac{1}{2r}\right)\left[m_\l^p\int\limits_{\O^-}\int\limits_{\O^c}\frac{|u_\l^-(x)|^p}{|x-y|^{N+ps}}dydx+m_\l^p\int\limits_{\O^c}\int\limits_{\O^-}\frac{|u_\l^-(y)|^p}{|x-y|^{N+ps}}dydx\right]\\
			+& \l \left(\frac{1}{2r}-\frac{1}{q}\right)\left(l_\l^q |u_\l^+|_q^q +m_\l^q |u_\l^-|_q^q\right).
		\end{align}
		Since $l_\l,m_\l\leq1$, \begin{eqnarray}\nonumber\label{7.21}
			0\leq l_\l u_\l^+(x)+m_\l(-u_\l^-(y))\leq u_\l^+(x)-u_\l^-(y)\\
			\implies | l_\l u_\l^+(x)-m_\l u_\l^-(y)|^p \leq |u_\l^+(x)-u_\l^-(y)|^p.
		\end{eqnarray}
		Similarly we have \begin{equation}\label{7.22}
			|m_\l u_\l^-(x)-l_\l u_\l^+(y)|^p\leq |u_\l^-(x)-u_\l^+(y)|^p.
		\end{equation}
		Using \eqref{7.21} and \eqref{7.22} in \eqref{7.20}, we obtain
		\begin{align*}
			c_\l \leq&\left(\frac{1}{p}-\frac{1}{2r}\right)\left[\int\limits_{\O^+}\int\limits_{\O^+}\frac{|u_\l^+(x)-u_\l^+(y)|^p}{|x-y|^{N+ps}}dydx +\int\limits_{\O^-}\int\limits_{\O^-}\frac{|u_\l^-(x)-u_\l^-(y)|^p}{|x-y|^{N+ps}}dydx\right]\\ &+\left(\frac{1}{p}-\frac{1}{2r}\right)\left[\int\limits_{\O^+}\int\limits_{\O^-}\frac{| u_\l^+(x)- u_\l^-(y)|^p}{|x-y|^{N+ps}}dydx+\int\limits_{\O^-}\int\limits_{\O^+}\frac{| u_\l^+(y)- u_\l^-(x)|^p}{|x-y|^{N+ps}}dydx\right]\\ 
			&+\left(\frac{1}{p}-\frac{1}{2r}\right)\left[\int\limits_{\O^+}\int\limits_{\O^c}\frac{|u_\l^+(x)|^p}{|x-y|^{N+ps}}dydx+\int\limits_{\O^c}\int\limits_{\O^+}\frac{|u_\l^+(y)|^p}{|x-y|^{N+ps}}dydx\right]\\ 
			&+\left(\frac{1}{p}-\frac{1}{2r}\right)\left[\int\limits_{\O^-}\int\limits_{\O^c}\frac{|u_\l^-(x)|^p}{|x-y|^{N+ps}}dydx+\int\limits_{\O^c}\int\limits_{\O^-}\frac{|u_\l^-(y)|^p}{|x-y|^{N+ps}}dydx\right]\\
			&+ \l \left(\frac{1}{2r}-\frac{1}{q}\right)\left( |u_\l^+|_q^q + |u_\l^-|_q^q\right)\\ 
			=& \left(\frac{1}{p}-\frac{1}{2r}\right)\|u_\l\|^p+ \l \left(\frac{1}{2r}-\frac{1}{q}\right) |u_\l|_q^q\\ 
			\leq & \liminf_n\left(\left(\frac{1}{p}-\frac{1}{2r}\right)\|u_n\|^p+ \l \left(\frac{1}{2r}-\frac{1}{q}\right) |u_n|_q^q\right)\\
			=& \lim\limits_n \left(J_\l(u_n)-\frac{1}{2r}\langle J_\l^{\prime}(u_n),u_n \rangle \right)\\
			=&c_\l.
		\end{align*}
		Hence $l_\l=m_\l=1$ and $J_\l(u_\l)=c_\l$.
	\end{proof}
	\begin{lemma}
		We have	$u_\l$ is a critical point of $J_\l$, where $u_\l$ is obtained as in Lemma \ref{l7.6}.
	\end{lemma}
	\begin{proof}
		The proof follows exactly as in \cite[Theorem 1.3 (step 2)]{SX}. %We here give the proof for the sake of completeness.\\
	\end{proof}
	\begin{lemma}\label{8.8}
		For any $u \in M_\l$, there exists $\tilde{l}, \tilde{m} \in (0,1]$ such that $\tilde{l} u^+,~\tilde{m} u^- \in N_\l$.
	\end{lemma}
	\begin{proof}
		By Lemma \ref{l7.1} we know that there exists $\tilde{l},~\tilde{m}$ such that $\tilde{l} u^+,~\tilde{m} u^- \in N_\l.$ Next, we show that $\tilde{l} \in (0,1]$. The case for $\tilde{m}$ follows similarly. Since $ u \in M_\l$, we have $\langle J_\l^{\prime}(u),u^+\rangle=0$. This implies as in Lemma \ref{7.2}
		\begin{equation}\label{8.26}
			X^+(u)=\l |u^+|_q^q-Y^+(u).
		\end{equation} 
		%\begin{align}\label{8.26}\nonumber
		%	0=& \int\limits_Q \frac{(u(x)-u(y))^{p-1}(u^+(x)-u^+(y))}{|x-y|^{N+ps}}dydx-\int\limits_\O |u^+|^q dx\\ \nonumber
		%	&+\int\limits_\O\int\limits_\O\frac{|u(y)|^r|u^+(x)|r}{|x-y|^\a}dydx.\\ \nonumber
		%	=& \int\limits_{\O^+} \int\limits_{\O^+} \frac{u^+(x)-u^+(y)|^p}{|x-y|^{N+ps}}dydx + \int\limits_{\O^+}\int\limits_{\O^-}\frac{|u^+(x)-u^-(y)|^{p-1}u^+(x)}{|x-y|^{N+ps}}dydx\\ \nonumber
		%	&+ \int\limits_{\O^-}\int\limits_{\O^+}\frac{|u^+(y)-u^-(x)|^{p-1}u^+(y)}{|x-y|^{N+ps}}dydx + \int\limits_{\O^+}\int\limits_{\O^c} \frac{|u^+(x)|^p}{|x-y|^{N+ps}}dydx\\ \nonumber
		%	&+\int\limits_{\O^c}\int\limits_{\O^+} \frac{|u^+(y)|^p}{|x-y|^{N+ps}}dydx-\int\limits_\O |u^+|^qdx\\
		%	&+\int\limits_{\O^+}\int\limits_{\O^+}\frac{|u^+(y)|^r|u^+(x)|^r}{|x-y|^\a}dydx+ \int\limits_{\O^+} \int\limits_{\O^-}\frac{|u^-(y)|^r|u^+(x)|^r}{|x-y|^\a}dydx.
		%\end{align}
		Again since $\tilde{l} u^+ \in N_\l$, $\langle J_\l^\prime(\tilde{l}u^+),\tilde{l}u^+\rangle=0.$ This implies  $\langle J_\l^\prime(\tilde{l}u),\tilde{l}u^+\rangle\geq 0,$ and so
		\begin{equation}\label{8.27}
			\frac{1}{\tilde{l}^{2r-p}}X^+(u)\geq\l \tilde{l}^{q-2r}|u^+|_q^q-Y^+(u).
		\end{equation}
		%	\begin{align}\nonumber\label{8.27}
		%	0=& \tilde{l}^p\int\limits_Q\frac{|u^+(x)-u^+(y)|^p}{|x-y|^{N+ps}}dydx-\tilde{l}^q\int\limits|u^+|^qdx + \tilde{l}^{2r}\int\limits_{\O^+}\int\limits_{\O^+} \frac{|u^+(y)|^r|u^+(x)|^r}{|x-y|^\a}dydx\\ \nonumber
		%	=& \tilde{l}^p \int\limits_{\O^+}\int\limits_{\O^+}\frac{|u^+(x)-u^+(y)|^p}{|x-y|^{N+ps}}dydx+\tilde{l}^p\int\limits_{\O^+}\int\limits_{\O^-}\frac{|u^+(x)|^p}{|x-y|^{N+ps}}dydx\\ \nonumber
		%	&+\tilde{l}^p\int\limits_{\O^-}\int\limits_{\O^+}\frac{|u^+(y)|^p}{|x-y|^{N+ps}}dydx+\tilde{l}^p\int\limits_{\O^+}\int\limits_{\O^c}\frac{|u^+(x)|^p}{|x-y|^{N+ps}}dydx\\ \nonumber
		%	&+\tilde{l}^p\int\limits_{\O^c}\int\limits_{\O^+}\frac{|u^+(y)|^p}{|x-y|^{N+ps}}dydx-\tilde{l}^q\int\limits|u^+|^qdx + \tilde{l}^{2r}\int\limits_{\O^+}\int\limits_{\O^+} \frac{|u^+(y)|^r|u^+(x)|^r}{|x-y|^\a}dydx\\ \nonumber
		%	\leq& \tilde{l}^p \int\limits_{\O^+}\int\limits_{\O^+}\frac{|u^+(x)-u^+(y)|^p}{|x-y|^{N+ps}}dydx+ \tilde{l}^p\int\limits_{\O^+}\int\limits_{\O^-}\frac{|u^+(x)-u^-(y)|^{p-1}u^+(x)}{|x-y|^{N+ps}}dydx\\ \nonumber
		%	&+\tilde{l}^p\int\limits_{\O^-}\int\limits_{\O^+}\frac{|u^+(y)-u^-(x)|^{p-1}u^+(y)}{|x-y|^{N+ps}}dydx+\tilde{l}^p\int\limits_{\O^+}\int\limits_{\O^c}\frac{|u^+(x)|^p}{|x-y|^{N+ps}}dydx\\ \nonumber
		%	&+\tilde{l}^p\int\limits_{\O^c}\int\limits_{\O^+}\frac{|u^+(y)|^p}{|x-y|^{N+ps}}dydx-\tilde{l}^q\int\limits|u^+|^qdx + \tilde{l}^{2r}\int\limits_{\O^+}\int\limits_{\O^+} \frac{|u^+(y)|^r|u^+(x)|^r}{|x-y|^\a}dydx\\
		%	&+\tilde{l}^{2r}\int\limits_{\O^+}\int\limits_{\O^-} \frac{|u^-(y)|^r|u^+(x)|^r}{|x-y|^\a}dydx
		%\end{align}
		On subtracting \eqref{8.26} from \eqref{8.27}, we obtain
		\begin{equation}\label{8.28}
			\left(\frac{1}{\tilde{l}^{2r-p}}-1\right)X^+(u)\geq \l (\tilde{l}^{q-2r}-1)|u^+|_q^q.
		\end{equation}
		%\begin{align}\nonumber\label{8.28}
		%	(\tilde{l}^{q-2r}-1)\int\limits_O|u^+|^qdx\leq& \left(\frac{1}{\tilde{l}^{2r-p}}-1\right)\int\limits_{\O^+}\int\limits_{\O^+}\frac{|u^+(x)-u^+(y)|^p}{|x-y|^{N+ps}}dydx \\ \nonumber
		%	&+\left(\frac{1}{\tilde{l}^{2r-p}}-1\right)\int\limits_{\O^+}\int\limits_{\O^-}\frac{|u^+(x)-u^-(y)|^{p-1}u^+(x)}{|x-y|^{N+ps}}dydx\\ \nonumber
		%	&+\left(\frac{1}{\tilde{l}^{2r-p}}-1\right)\int\limits_{\O^-}\int\limits_{\O^+}\frac{|u^+(y)-u^-(x)|^{p-1}u^+(y)}{|x-y|^{N+ps}}dydx\\
		%	&+\left(\frac{1}{\tilde{l}^{2r-p}}-1\right)\left[\int\limits_{\O^+}\int\limits_{\O^c}\frac{|u^+(x)|^p}{|x-y|^{N+ps}}dydx+\int\limits_{\O^c}\int\limits_{\O^+}\frac{|u^+(y)|^p}{|x-y|^{N+ps}}dydx \right].
		%\end{align}
		Clearly $\tilde{l}>1$ gives a contradiction to \eqref{8.28}. Therefore, $\tilde{l} \in (0,1]$.\\
		This completes the proof.
	\end{proof}
	\textbf{Proof of Theorem \ref{t2.11}}:
	\begin{proof}
		Combining Lemmas \ref{l7.2} to \ref{l7.6}, we conclude that $u_\l$ is a least energy sign-changing solution of $(P_\l)$. Also by Lemma \ref{8.8} there exist $\tilde{l}_\l$, $\tilde{m}_\l \in (0,1]$ such that $\tilde{l}_\l u_\l^+$, $\tilde{m}_\l u_\l^- \in N_\l$. Now using $\|u_\l\|^p > \|u_\l^+\|^p +\|u^-_\l\|^p$, we have
		\begin{align*}
			c_\l=& J_\l(u_\l)\\
			=& J_\l(u_\l)-\frac{1}{2r}\langle J_\l^{\prime}(u_\l),u_\l\rangle\\
			=&	 \left(\frac{1}{p}-\frac{1}{2r}\right)\|u_\l\|^p+\left(\frac{1}{2r}-\frac{1}{q}\right)|u_\l|_q^q\\
			>&\left(\frac{1}{p}-\frac{1}{2r}\right)\|u^+_\l\|^p+\left(\frac{1}{2r}-\frac{1}{q}\right)|u^+_\l|_q^q+ \left(\frac{1}{p}-\frac{1}{2r}\right)\|u^-_\l\|^p+\left(\frac{1}{2r}-\frac{1}{q}\right)|u^-_\l|_q^q\\
			\geq&\left(\frac{1}{p}-\frac{1}{2r}\right)\|\tilde{l}_\l u^+_\l\|^p+\left(\frac{1}{2r}-\frac{1}{q}\right)|\tilde{l}_\l u^+_\l|_q^q + \left(\frac{1}{p}-\frac{1}{2r}\right)\|\tilde{m}_\l u^-_\l\|^p \\ & \quad \quad +\left(\frac{1}{2r}-\frac{1}{q}\right)|\tilde{m}_\l u^-_\l|_q^q\\
			=& J_\l(\tilde{l}u^+_\l)-\frac{1}{2r}\langle J_\l^{\prime}(\tilde{l}u^+_\l),\tilde{l}u^+_\l\rangle+J_\l(\tilde{l}u^+_\l)-\frac{1}{2r}\langle J_\l^{\prime}(\tilde{m}u^-_\l),\tilde{m}u^-_\l\rangle\\
			\geq &2m_\l.
		\end{align*}
		This completes the proof.
	\end{proof}
	\textbf{Funding:} The first author thanks the CSIR(India) for financial support in the form of a Junior Research Fellowship, Grant Number $09/086(1406)/2019$-EMR-I. The second author is funded by IFCAM (Indo-French Centre for Applied Mathematics) IRL CNRS 3494.

\end{document}